\documentclass{amsart}

\usepackage{amsmath,amsfonts,amssymb,amsthm,graphicx,cite}
\usepackage{amscd}
\usepackage{mathrsfs}
\usepackage{lineno}
\usepackage{float}
\usepackage{caption}
\usepackage{subcaption}

\newcommand{\enm}[1]{e(n,m;{#1})}
\newcommand{\snm}[1]{s(n,m;{#1})}
\newcommand{\SV}{\mathrm{SV}}

\newcommand{\tet}{\theta}
\newcommand{\ii}{\mathrm{i}}
\newcommand{\dd}{d}
\newcommand{\ee}{\mathrm{e}}

\newcommand{\cP}{\mathcal{P}}

\newtheorem{theorem}{Theorem}[section]

\newtheorem{lemma}{Lemma}[section]
\newtheorem{proposition}{Proposition}[section]
\theoremstyle{remark}
\newtheorem{remark}{Remark}[section]
\newtheorem{definition}{Definition}[section]

\newcommand{\Z}{{\mathbb Z}}

\newcommand{\C}{{\mathbb C}}
\newcommand{\R}{{\mathbb R}}

\newcommand{\T}{{\mathbb T}}

\newcommand{\cM}{{\mathcal M}}

\DeclareMathOperator{\spn}{span}

\begin{document}

\title[Super-Macdonald polynomials]{Super-Macdonald polynomials: orthogonality and Hilbert space interpretation}

\author{Farrokh Atai}
\address{Department of Mathematics, Kobe University, Rokko, Kobe 657-8501, Japan}
\email{farrokh@math.kobe-u.ac.jp}

\author{Martin Halln\"as}
\address{Department of Mathematical Sciences, Chalmers University of Technology and the University of Gothenburg, SE-412 96 Gothenburg, Sweden}
\email{hallnas@chalmers.se}

\author{Edwin Langmann}
\address{Department of Physics, KTH Royal Institute of Technology, SE-106 91 Stockholm, Sweden}
\email{langmann@kth.se}

\date{\today}

\begin{abstract}
The super-Macdonald polynomials, introduced by Sergeev and Veselov \cite{SV09a}, generalise the Macdonald polynomials to (arbitrary numbers of) two kinds of variables, and they are eigenfunctions of the deformed Macdonald-Ruijsenaars operators introduced by the same authors in \cite{SV04}.
 
We introduce a Hermitian form on the algebra spanned by the super-Macdonald polynomials, prove their orthogonality, compute their (quadratic) norms explicitly, and establish a corresponding Hilbert space interpretation of the super-Macdonald polynomials and deformed Macdonald-Ruijsenaars operators. This allows for a quantum mechanical interpretation of the models defined by the deformed Macdonald-Ruijsenaars operators. Motivated by recent results in the nonrelativistic ($q\to 1$) case, we propose that these models describe the particles and anti-particles of an underlying relativistic quantum field theory, thus providing a natural generalisation of the trigonometric Ruijsenaars model.
\end{abstract}

\maketitle

\tableofcontents

\section{Introduction}
\label{sec:intro}
As is well-known, the Macdonald polynomials \cite{Mac95} can be viewed as eigenfunctions of a commuting family of difference operators associated with a relativistic generalisation of the integrable quantum Calogero-Moser-Sutherland systems of trigonometric A-type \cite{Rui87}. Such relativistic quantum systems were originally conceived by Ruijsenaars as an integrable quantum mechanical description of a relativistic quantum field theory in 1+1 spacetime dimensions known as the quantum sine-Gordon theory, restricted to sectors where the particle number is fixed \cite{RS86,Rui01}.

While the standard Ruijsenaars systems account for one particle type, a relativistic quantum field theory typically has two kinds of particle: particles and anti-particles.
This strongly suggests to us that Ruijsenaars' systems should have generalisations allowing for two particle types, and we propose that, in the trigonometric regime, such a generalisation is given by the so-called deformed Macdonald-Ruijsenaars operators $\cM_{n,m;q,t}$ and $\cM_{n,m;q^{-1},t^{-1}} $ (specified in \eqref{cMnm1} below), and their joint eigenfunctions, the super-Macdonald polynomials, introduced and studied by Sergeev and Veselov \cite{SV04,SV09a}. The super-Macdonald polynomials $SP_\lambda((x_1,\ldots,x_n),(y_1,\ldots,y_m);q,t)$ depend on arbitrary numbers, $n$ and $m$, of {\em two} types of variables, $x_i$ and $y_j$, and we expect that these two variable types correspond to particles and anti-particles in an underlying quantum field theory.

In any quantum mechanical model, there is a scalar product providing the space of wave functions with a Hilbert space structure, and this structure is essential for the physical interpretation of the model. For the trigonometric Ruijsenaars systems and the Macdonald polynomials such a Hilbert space structure is provided by the scalar product denoted as $\langle\cdot,\cdot\rangle^\prime_n$ in Macdonald's book \cite{Mac95}. In particular, with respect to this scalar product, the commuting family of difference operators alluded to above, which include operators that define the Hamiltonian and momentum operator in the model, are self-adjoint and the Macdonald polynomials form an orthogonal system with explicitly known Hilbert space norms \cite{Mac95}. By contrast, for the deformed Macdonald-Ruijsenaars operators and the super-Macdonald polynomials, such a Hilbert space structure has been missing. Our main purpose with this paper is to provide this missing Hilbert space structure and thereby substantiate our proposal, as formulated above. Moreover, recent quantum field theory results in the nonrelativistic case \cite{AL17,BLL20}, discussed in Section~\ref{sec:final}, provide further support in favour of our proposal.

To describe our results in more detail, we recall from \cite{SV09a} that the super-Macdonald polynomials are joint eigenfunctions of a large commutative algebra of difference-operators, containing the so-called deformed Macdonald-Ruijsenaars operator introduced in \cite{SV04}:\footnote{Note that we use somewhat different conventions --- see Appendix~\ref{app:SV} for how the conventions are related.}
\begin{equation}
\label{cMnm1} 
\begin{split} 
\cM_{n,m;q,t} = & \frac{t^{1-n}}{1-q} \sum_{i=1}^n A_i(x,y;q,t)\big(T_{q,x_i}-1\big)\\
&+ \frac{q^{m-1}}{1-t^{-1}} \sum_{j=1}^m B_j(x,y;q,t)\big(T_{t^{-1},y_j}-1\big)
\end{split} 
\end{equation}
with coefficients
\begin{equation} 
\label{cMnm2} 
\begin{split}
A_i(x,y;q,t) &= \prod_{i^\prime \neq i}^n \frac{tx_i - x_{i^\prime}}{x_i- x_{i^\prime}}\cdot \prod_{j=1}^m \frac{t^{1/2} x_i - q^{1/2} y_j}{t^{1/2} x_i - q^{-1/2}y_j},\\
B_j(x,y;q,t) &= \prod_{j^\prime \neq j}^m \frac{q^{-1}y_j - y_{j^\prime}}{y_j - y_{j^\prime}}\cdot \prod_{i=1}^n \frac{q^{-1/2} y_j - t^{-1/2} x_i}{q^{-1/2} y_j - t^{1/2} x_i},
\end{split} 
\end{equation} 
and where $T_{q,x_i}$ and $T_{t^{-1},y_j}$ act on functions 
$f(x,y)$ of $x=(x_1,\ldots,x_n)\in\mathbb{C}^n$ and $y=(y_1,\ldots,y_m)\in\mathbb{C}^m$ by shifting $x_i\to qx_i$ and $y_j\to t^{-1}y_j$, respectively, while leaving the remaining variables unaffected.

For our Hilbert space results, it will be important to restrict attention to parameter values
\begin{equation*}
\label{qt} 
0<q <1,\ \ \  0<t<1.
\end{equation*}
However, as discussed briefly in the final paragraph of Section~\ref{sec:final}, some of our results extend analytically to complex $q$ and $t$ (with modulus in $(0,1)$).

Deformed Macdonald-Ruijsenaars operators first appeared in the $m=1$ case in work by Chalykh \cite{Cha97,Cha00}. Further examples, including deformed Koornwinder operators, were later obtained and studied by Feigin \cite{Fei05}, Sergeev and Veselov \cite{SV09b} and Feigin and Silantyev \cite{FS14}.

Taking $m=0$, the operator given by \eqref{cMnm1}--\eqref{cMnm2} reduces to
\begin{equation} 
\label{cMn0} 
\cM_{n;q,t} = \frac{t}{1-q}\cdot t^{-n}\sum_{i=1}^n \prod_{i^\prime \neq i}^n \frac{tx_i - x_{i^\prime}}{x_i- x_{i^\prime}} \big(T_{q,x_i}-1\big),
\end{equation} 
and the super-Macdonald polynomials reduce to the ordinary (monic symmetric) Macdonald polynomials $P_\lambda((x_1,\ldots,x_n);q,t)$. Note that, up to the overall factor $t/(1-q)$, $\cM_{n;q,t}$ \eqref{cMn0} coincides with the operator $E_n$ in Madonald's book \cite[Section~VI.4]{Mac95}. Moreover, $\cM_{n;q,t}$ is closely related to the trigonometric limit of the elliptic operator $\widehat{S}_1$  introduced by Ruijsenaars in \cite{Rui87}. The precise relationship, which was first observed by Koornwinder (in unpublished notes), is, e.g., detailed in \cite[Section 5.2]{vDie95} and \cite[Section 5.1]{Has97}. We recall that $\widehat{S}_1\pm \widehat{S}_{-1}$, where $\widehat{S}_{-1}$ is similarly related to $\cM_{n;q^{-1},t^{-1}}$, essentially amount to the Hamiltonian and momentum operator, respectively, in Ruijsenaars' model. This state of affairs suggests to us that the super-Macdonald operators define a quantum mechanical model generalising the trigonometric Ruijsenaars model by allowing two kinds of particles; see Appendix~\ref{app:relativistic} for a proof of the relativistic invariance of this generalized model. However, such an interpretation requires a compatible Hilbert space structure.

As is well known, the Macdonald polynomials form an orthogonal system on the $n$-dimensional torus $\T^n\equiv \T_1^n$, where
\begin{equation}
\label{tori}
\T_\xi^n = \big\{x=(x_1,\ldots,x_n)\in\C^n\mid |x_i|=\xi\ \ (i=1,\ldots,n)\big\}\ \ \ (\xi > 0),
\end{equation}
with respect to the weight function
\begin{equation}
\label{Deln}
\Delta_n(x;q,t) = \prod_{1\leq i\neq j\leq n}\frac{(x_i/x_j;q)_\infty}{(tx_i/x_j;q)_\infty},
\end{equation}
where $(a;q)_\infty=\prod_{k=0}^\infty (1-aq^k)$ is the usual $q$-Pochhammer symbol. Moreover, the corresponding (quadratic) norms are given by remarkably simple and explicit formulas \cite[Section VI.9]{Mac95}; see \eqref{prod0}--\eqref{Nn}.

These orthogonality results, together with the corresponding Hilbert space structure, entail a natural quantum mechanical interpretation of 
the Macdonald polynomials $P_\lambda((x_1,\ldots,x_n);q,t)$ and the commuting Macdonald-Ruijsenaars operators $\cM_{n;q,t}$ and $\cM_{n;q^{-1},t^{-1}}$; this is the trigonometric Ruijsenaars model.

In this paper, we obtain analogous results for the the super-Macdonald polynomials $SP_\lambda(x,y;q,t)$. More specifically, we establish orthogonality relations with respect to a sesquilinear form given by
\begin{equation} 
\label{prod0S}
\begin{split} 
\langle P,Q\rangle^\prime_{n,m;q,t} = \frac1{n!m!}\int_{\T_\xi^n} \frac{\dd x_1}{2\pi\ii x_1}\cdots  \frac{\dd x_n}{2\pi\ii x_n}  \int_{\T_{\xi^\prime}^m} \frac{\dd y_1}{2\pi\ii y_1}\cdots  \frac{\dd y_m}{2\pi\ii y_m}
\\ \times  \Delta_{n,m}(x,y;q,t) P(x,y)\overline{Q(\bar{x}^{-1},\bar{y}^{-1})},
\end{split}
\end{equation} 
with weight function
\begin{equation}
\label{wghtFunc}
\Delta_{n,m}(x,y;q,t) = \frac{\Delta_n(x;q,t)\Delta_m(y;t,q)}{\prod_{i=1}^n\prod_{j=1}^m(1-q^{-1/2}t^{1/2}x_i/y_j)(1-q^{-1/2}t^{1/2}y_j/x_i)},
\end{equation}
and where $P,Q$ are polynomials in the space spanned by the super-Macdonald polynomials, the bar denotes complex conjugation and
\begin{equation}
\label{barinv}
\bar{x}^{-1}:=(1/\bar x_1,\ldots,1/\bar x_n),\ \ \ \bar{y}^{-1}:=(1/\bar y_1,\ldots,1/\bar y_m).
\end{equation}
Furthermore, in order to ensure that we avoid the poles of the weight function, we integrate $x$ and $y$ over tori $\T_\xi^n$ and $\T_{\xi^\prime}^m$ with radii $\xi,\xi^\prime>0$ that are sufficiently separated. 
Our main results are:
\begin{enumerate}
\item[(I)] The expression \eqref{prod0S} defines a Hermitian product that is independent of $\xi,\xi^\prime>0$ provided $|\log(\xi/\xi^\prime)|>\frac12|\log(q/t)|$,
\item[(II)] the orthogonality relations $\langle SP_\lambda,SP_\mu\rangle^\prime_{n,m;q,t}=0$ hold true for all $\lambda\neq\mu$,
\item[(III)] the (squared) norms $\langle SP_\lambda,SP_\lambda\rangle^\prime_{n,m;q,t}$ are given by the simple and explicit formulas \eqref{zNnm}--\eqref{nzNnm}.
\end{enumerate}

\begin{remark} 
\label{remarkxi}
The attentive reader might wonder why we do not simply integrate over $\T^n\times\T^m$ in \eqref{prod0S}  since, clearly, poles in the denominator of \eqref{wghtFunc} would also be avoided by choosing $\xi=\xi^\prime=1$. 
This can be readily understood in the simplest non-trivial case $n=m=1$ since, in this case, the integral in \eqref{prod0S} can be easily computed; see Appendix~\ref{app:nm1} for details. 
One finds that the integral is the same for $\xi\gg\xi^\prime$ and $\xi^\prime\gg\xi$, but the integral for $\xi=\xi^\prime$ differs by a non-trivial residue term which spoils our orthogonality results, as described above.  
\end{remark}

Remarkably, even though we are working with a complex-valued weight function (since the denominator in \eqref{wghtFunc} is only real if $\xi=\xi^\prime$), we find that all norms are given by non-negative real numbers. 
In addition, the super-Macdonald polynomials with non-zero norms are characterised by the simple condition $\lambda_{n}\geq m\geq\lambda_{n+1}$.
As discussed in Section~\ref{sec:hil}, the product $\langle\cdot,\cdot\rangle^\prime_{n,m;q,t}$ therefore provides the space spanned by the super-Macdonald polynomials with non-zero norm with a Hilbert space structure allowing for a quantum mechanical interpretation of the model defined by the commuting deformed Macdonald-Ruijsenaars operators $\cM_{n,m,q,t}$ and $\cM_{n,m,q^{-1},t^{-1}}$. 

The results in this paper can be considered as natural $q$-deformations of the orthogonality relations and norm formula we obtained in \cite{AHL19} for the super-Jack polynomials. As compared to {\em loc.~cit.}, significant simplifications occur: Since the eigenvalues of $\cM_{n,m;q,t}$ separate the super-Macdonald polynomials $SP_\lambda$, there is no need to involve higher order eigenoperators; and the fact that $\Delta_{n,m}$ is a meromorphic function simplifies arguments involving contour deformations.

Our plan is as follows. 
In Section~\ref{sec:preliminary}, we briefly review known facts about the Macdonald functions (Section~\ref{sec:MF}) and super-Macdonald polynomials (Section~\ref{sec:sMP}) that we need. 
Our results can be found in Section~\ref{sec:results}: a precise formulation of our orthogonality result is given in Theorem~\ref{thm} (Section~\ref{sec:orth}), followed by a discussion of the Hilbert space interpretation of the super-Macdonald polynomials  suggested by this (Section~\ref{sec:hil}).   
The proof of Theorem~\ref{thm} is given in  Section~\ref{sec:proof}. 
We conclude with a short discussion of research questions motivated by our results in Section~\ref{sec:final}. 
Three appendices explain how the conventions on super-Macdonald polynomials we use are related to the ones of Sergeev and Veselov \cite{SV09a} (Appendix~\ref{app:SV}), prove the relativistic invariance of the generalized Ruijsenaars model  (Appendix~\ref{app:relativistic}), give proof details to make this paper self-contained (Appendix~\ref{app:proof}), and shortly discuss the special case $n=m=1$ (Appendix~\ref{app:nm1}).

\subsection*{Notation} 
We denote as $\cP$ the space of all partitions, i.e., $\lambda\in\cP$ means that $\lambda=(\lambda_1,\lambda_2,\ldots)$ with integers $\lambda_i\geq 0$ satisfying $\lambda_i\geq \lambda_{i+1}$, $i=1,2,\ldots$, and only finitely many $\lambda_i$'s non-zero; the non-zero $\lambda_i$'s are called {\em parts of $\lambda$}, and partitions differing only by a string of zeros at the end are not distinguished.
For any partition $\lambda$, $\ell(\lambda)$ is the number of parts of $\lambda$, and $|\lambda|$ is the sum of its parts; $\ell(\lambda)$ and $|\lambda|$ are called {\em length} and {\em weight} of $\lambda$, respectively. Moreover, for $\lambda\in\cP$, $\lambda'$ denotes the conjugate of $\lambda$ (so that the Young diagrams of $\lambda$ and $\lambda'$ are transformed into each other by reflection in the main diagonal). We also recall the definition of the dominance partial ordering on the set of partitions of a fixed weight: 
for $\lambda,\mu\in\cP$ such that $|\lambda|=|\mu|$, 
\begin{equation*}
\lambda\leq \mu \Leftrightarrow \sum_{i=1}^j \lambda_i \leq  \sum_{i=1}^j \mu_i \quad (j=1,2,\ldots). 
\end{equation*} 
In addition, for $\lambda,\mu\in\cP$, $\mu\subseteq\lambda$ is short for $\mu_i\leq \lambda_i$ for all $i$, and $\lambda\cup\mu$ denotes the partition obtained by merging and re-ordering the parts of $\lambda$ and $\mu$. 

For $N\in\Z_{\geq 2}$, $1\leq i\neq j\leq N$ means $i,j=1,\ldots,N$, $i\neq j$, and we write 
$$\prod_{i\neq j}^N\ \ \ \text{ short for } \ \ \ \prod_{\substack{1\leq i\leq N\\i\neq j}}.$$ 

For $z\in\C$, $\bar z$ is the complex conjugate of $z$. For $z=(z_1,\ldots,z_N)\in\C^N$ with $N\in\Z_{\geq 2}$, $z^{-1}$, $\bar z$ and $\bar z^{-1}$ are short for $(1/z_1,\ldots,1/z_N)$, $(\bar z_1,\ldots,\bar z_N)$ and $(1/\bar z_1,\ldots,1/\bar z_N)$, respectively. We write $\ii:=\sqrt{-1}$ and $\C^*:=\C\setminus\{0\}$.

\section{Prerequisites}
\label{sec:preliminary} 
We collect definitions and results we need, following Macdonald \cite{Mac95} in Section~\ref{sec:MF} and Sergeev and Veselov \cite{SV09a} in Section~\ref{sec:sMP}. 

\subsection{Macdonald functions}
\label{sec:MF}
Unless mentioned otherwise, $\lambda,\mu$ are arbitrary partitions.

\subsubsection{Ring of symmetric functions}
We consider the complex vector space $\Lambda=\Lambda_{\C}$ of symmetric functions in infinitely many variables $x=(x_1,x_2,\ldots)$ (we work over $\C$ since we are motivated by quantum mechanics). 
It can be defined as the space of all finite linear combinations, with complex coefficients, of  the symmetric monomial functions $m_\lambda$, labeled by partitions $\lambda$, and defined as follows:
\begin{equation}
\label{mlam} 
m_\lambda(x) := \sum_a x_1^{a_1}x_2^{a_2}\cdots,
\end{equation} 
where the sum is over all distinct permutations $a=(a_1,a_2,\ldots)$ of $\lambda=(\lambda_1,\lambda_2,\ldots)$.
Thus, the symmetric monomial functions constitute a (vector space) basis in $\Lambda$ labeled by partitions. 
Another such basis is given by the products
\begin{equation} 
\label{plam}
p_\lambda := \prod_{i=1}^{\ell(\lambda)} p_{\lambda_i}\quad (\lambda\in\cP)
\end{equation} 
of the Newton sums 
\begin{equation}
\label{pr}  
p_r(x) := \sum_{i\geq 1} x_i^r\quad (r\in\Z_{\geq 1}) . 
\end{equation} 
The space $\Lambda$ has a natural ring structure and, as such, is freely generated by the Newton sums $p_r$, $r\in\Z_{\geq 1}$. 

\subsubsection{Macdonald functions}
The space $\Lambda$ becomes a (pre-)Hilbert space when equipped with the scalar product $\langle\cdot,\cdot\rangle_{q,t}$ characterised by linearity in its first (and antilinearity in its second) argument and\footnote{To avoid possible confusion, we stress that this product is different from the one allowing for a quantum mechanical interpretation of the Macdonald polynomials.}
\begin{equation} 
\label{prodMac}
\langle p_\lambda,p_\mu\rangle_{q,t} = \delta_{\lambda\mu} z_\lambda \prod_{i=1}^{\ell(\lambda)}\frac{1-q^{\lambda_i}}{1-t^{\lambda_i}}, 
\end{equation} 
where $z_\lambda := \prod_{i=1}^{\lambda_1} i^{m_i}m_i!$ with $m_i=m_i(\lambda)$ the number of parts of $\lambda$ equal to $i$ (setting $i^0 0!=1$), and  $\delta_{\lambda\mu}$ the Kronecker delta.  

As proved in \cite{Mac95}, the Macdonald functions $P_\lambda=P_\lambda(x;q,t)\in \Lambda$, $\lambda\in\cP$, can be defined by the following two conditions: triangular structure, 
\begin{subequations} 
\begin{equation*} 
P_\lambda = m_\lambda  + \sum_{\mu<\lambda} u_{\lambda\mu} m_\mu 
\end{equation*} 
for certain coefficients $u_{\lambda\mu}=u_{\lambda\mu}(q,t)$, and orthogonality, 
\begin{equation*} 
\langle P_\lambda,P_\mu\rangle_{q,t} = 0\quad (\lambda\neq \nu).
\end{equation*} 
\end{subequations} 

It is known that the Macdonald functions $P_\lambda$ are eigenfunctions of the inverse limit $\cM_{q,t}$ of the operators $\cM_{n;q,t}$ \eqref{cMn0}:  
\begin{equation}
\label{PeigEq}
\cM_{q,t}P_\lambda(x;q,t) = d_\lambda(q,t)P_\lambda(x;q,t), 
\end{equation}
and the corresponding eigenvalues are given by
\begin{equation}
\label{dlam}
d_\lambda(q,t) = \sum_{i\geq 1}t^{1-i}\frac{q^{\lambda_i}-1}{1-q}.
\end{equation}
Moreover, the Macdonald functions are known to be invariant under $(q,t)\to(q^{-1},t^{-1})$, i.e., 
\begin{equation} 
\label{qinvtinv} 
P_\lambda(x;q^{-1},t^{-1})=P_\lambda(x;q,t). 
\end{equation} 

We recall the definition of the dual Macdonald functions $Q_\lambda=Q_\lambda(x;q,t)$: 
\begin{equation} 
\label{Q} 
Q_\lambda = b_\lambda P_\lambda
\end{equation} 
with $b_\lambda=b_\lambda(q,t):=\langle P_\lambda,P_\lambda\rangle^{-1}$ given by  
\begin{equation}
\label{blam} 
b_\lambda(q,t) = \prod_{(j,k)\in\lambda}\frac{1-q^{\lambda_j-k}t^{\lambda_k'-j+1}}{1-q^{\lambda_j-k+1}t^{\lambda_k'-j}} = \frac1{b_{\lambda^\prime}(t,q)}, 
\end{equation}
where the product is over all $(j,k)$ such that $j=1,\ldots,\ell(\lambda)$ and $k=1,\ldots,\lambda_j$. 
The two kinds of Macdonald functions obey $\langle P_\lambda,Q_\lambda\rangle = \delta_{\lambda\mu}$, and they are related by the Macdonald involution
\begin{equation} 
\label{omqt}
\omega_{q,t}: \Lambda\to\Lambda,\ p_r\mapsto (-1)^{r-1}\frac{1-q^r}{1-t^r}p_r\quad (r\in\Z_{\geq 1}) 
\end{equation} 
as follows, 
\begin{equation} 
\label{duality} 
\omega_{q,t}(P_\lambda(x;q,t)) = Q_{\lambda'}(x;t,q).
\end{equation}

We also make use of the fact that the dual Macdonald functions $Q_\lambda$, like the Macdonald functions $P_\lambda$, are homogenous of degree $|\lambda|$: 
\begin{equation}
\label{homQ} 
Q_\lambda(sx;q,t)= s^{|\lambda|}Q_\lambda(x;q,t) \quad (s\in\C^*). 
\end{equation} 

\subsubsection{Skew functions} 
Let $f^{\lambda}_{\mu\nu}=f^{\lambda}_{\mu\nu}(q,t):=\langle Q_\lambda,P_\mu P_\nu\rangle_{q,t}$. Then the skew functions $P_{\lambda/\mu}\in\Lambda$ can be  defined by
\begin{equation} 
\label{Pskew} 
P_{\lambda/\mu}(x;q,t) : = \sum_{\nu} f^{\lambda'}_{\mu'\nu'}(t,q) P_\nu(x;q,t)
\end{equation} 
with $P_\nu(x;q,t)$ the Macdonald functions. It is well-known that $f^{\lambda'}_{\mu'\nu'}$ is non-zero only if $\mu\subseteq\lambda$, $\nu\subseteq\lambda$, and $|\mu|+|\nu|=|\lambda|$. 

These skew functions are homogenous of degree $|\lambda|-|\mu|$: 
\begin{equation} 
\label{homP} 
P_{\lambda/\mu}(sx;q,t)= s^{|\lambda|-|\mu|}P_{\lambda/\mu}(x;q,t) \quad (s\in\C^*), 
\end{equation} 
and they appear in the following expansion of Macdonald functions $P_\lambda(z;q,t)$ for variables $z=(x,y)$ obtained by merging two infinite sets of variables $x=(x_1,x_2,\ldots)$ and $y=(y_1,y_2,\ldots)$:   
\begin{equation} 
\label{PPQ} 
P_\lambda(x,y;q,t) = \sum_{\mu\subseteq\lambda} P_{\lambda/\mu}(x;q,t)Q_\mu(y;q,t). 
\end{equation} 

In the following Lemma, we state a well-known technical result that we need.  
\begin{lemma} 
\label{lem:f} 
The coefficients $f^{\lambda^\prime}_{\mu^\prime\nu^\prime}(t,q)$ in \eqref{Pskew} are non-zero only if 
\begin{equation*} 
\mu\cup\nu\leq\lambda\leq\mu+\nu, 
\end{equation*}  
and in the extremal cases they are given by 
\begin{equation*} 
f^{(\mu\cup\nu)^\prime}_{\mu^\prime\nu^\prime}(t,q)=1,\ \ \ f^{(\mu+\nu)^\prime}_{\mu^\prime\nu^\prime}(t,q)=\frac{b_\mu(q,t)b_\nu(q,t)}{b_{\mu+\nu}(q,t)}.
\end{equation*} 
\end{lemma} 
(For the convenience of the reader, we give a proof in Appendix~\ref{app:f}.) 

\subsubsection{Macdonald polynomials}
\label{sec:MP}
The Macdonald polynomials $P_\lambda((x_1,\ldots,x_n);q,t)$ are obtained from the Macdonald functions $P_\lambda((x_1,x_2,\ldots);q,t)$ by setting $x_{i}=0$ for all $i>n$, and similarly for $Q_\lambda$ and $P_{\lambda/\mu}$.

It is know that $P_\lambda((x_1,\ldots,x_n);q,t)$ is non-zero only for partitions $\lambda=(\lambda_1,\ldots,\lambda_n)$ of length less or equal to $n$. 
Moreover, as already discussed in the introduction, the Macdonald polynomials are orthogonal with respect to the following scalar product, 
\begin{equation} 
\label{prod0} 
\langle P,Q\rangle^\prime_{n;q,t} :=  \int_{\T^n} \frac{dx_1}{2\pi\ii x_1} \cdots  \frac{dx_n}{2\pi\ii x_n}  \Delta_n(x;q,t) P(x)\overline{Q(x)} 
\end{equation} 
for $P,Q$ symmetric polynomials in the variables $x=(x_1,\ldots,x_n)\in\C^n$, $\T^n=\T_1^n$ as in \eqref{tori}, and $\Delta_n(x;q,t)$ in \eqref{Deln}: for all $P_\lambda=P_\lambda(x;q,t)$ with $x=(x_1,\ldots,x_n)\in\C^n$ and $\lambda=(\lambda_1,\ldots,\lambda_n)$,  
\begin{equation} 
\label{orth0} 
\langle P_\lambda,P_\mu\rangle^\prime_{n;q,t} = \delta_{\lambda\mu}N_n(\lambda;q,t),
\end{equation} 
where 
\begin{equation} 
\label{Nn} 
N_n(\lambda;q,t)=  \prod_{1\leq i<j\leq n} \frac{(q^{\lambda_i-\lambda_j}t^{j-i};q)_\infty (q^{\lambda_i-\lambda_j+1}t^{j-i};q)_\infty}{(q^{\lambda_i-\lambda_j}t^{j-i+1};q)_\infty(q^{\lambda_i-\lambda_j+1}t^{j-i+1};q)_\infty}.
\end{equation} 
We also need
\begin{equation} 
\label{box} 
(x_1\ldots x_n)^k P_\lambda((x_1,\ldots,x_n);q,t) = P_{\lambda+(k^n)}((x_1,\ldots,x_n);q,t) \quad (k\in\Z_{\geq 0})
\end{equation} 
where $\lambda+(k^n)=(\lambda_1+k,\ldots,\lambda_n+k)$. 

\subsection{Super-Macdonald polynomials}
\label{sec:sMP}
Following Sergeev and Veselov \cite{SV09a}, we define $\Lambda_{n,m;q,t}$  as the algebra of complex polynomials $P(x,y)$ in $n+m$ variables $(x,y)=(x_1,\ldots,x_n,y_1,\ldots,y_m)\in\C^n\times\C^m$ that are symmetric in each set of variables separately, i.e., 
\begin{equation} 
\label{Lamnm} 
P(\sigma x;\tau y) = P(x,y)\ \ \ ((\sigma,\tau)\in S_n\times S_m)
\end{equation} 
where $S_n$ is the group of permutations of $n$ objects, and, furthermore, that satisfy the symmetry conditions\footnote{See Appendix~\ref{app:SV} for details on how our conventions are related to the ones in \cite{SV09a}.}
\begin{equation}
\label{qinv}
\big(T_{q,x_i}-T_{t^{-1},y_j}\big)P(x,y) = 0\  \text{ at } \  q^{1/2}x_i = t^{-1/2}y_j \ \ \  (\forall i,j). 
\end{equation}
This algebra, $\Lambda_{n,m;q,t}$, is generated by the following {\em deformed Newton sums}, 
\begin{equation}
\label{pnmr}
p_{r}(x,y;q,t)= \sum_{i=1}^n x_i^r -\frac{q^{r/2}-q^{-r/2}}{t^{r/2}-t^{-r/2}}  \sum_{k=1}^my_j^r \ \ \ (r\in\Z_{\geq 1}) 
\end{equation}
for $(x,y)\in\C^n\times\C^m$ \cite[Theorem 5.8]{SV09a}. 

\begin{remark} 
Many results in \cite{SV09a} require a restriction to so-called non-special parameters $q,t$, i.e., $q^{i}t^{j}\neq 1$ for all $i,j\in\Z_{\geq 0}$ such that $i+j\geq 1$; see e.g.\ \cite[Theorem 5.8]{SV09a}.\footnote{Note that our $t$ is $t^{-1}$ in \cite{SV09a}.} 
However, since we assume $0<q,t<1$, we can ignore this restriction. 
\end{remark} 

The super-Macdonald polynomials were defined in \cite{SV09a} as the image of the Macdonald functions $P_\lambda$ under the homomorphism 
\begin{equation}
\label{varphi}
\varphi_{n,m;q,t}: \Lambda\to\Lambda_{n,m;q,t},\ p_r((x_1,x_2,\ldots))\mapsto p_r((x_1,\ldots,x_n),(y_1,\ldots,y_m);q,t). 
\end{equation}
Thus, if $c_{\lambda\mu}(q,t)$ are the coefficients of the Macdonald polynomials defined by the expansion
\begin{equation} 
\label{SPdef1} 
P_\lambda(x;q,t)=\sum_{\mu} c_{\lambda\mu}(q,t) p_\mu(x)\quad (x=(x_1,x_2,\ldots)), 
\end{equation} 
then 
\begin{equation} 
\label{SPdef2} 
SP_\lambda(x,y;q,t)=\sum_{\mu} c_{\lambda\mu}(q,t) p_\mu(x,y;q,t)\quad ((x,y)\in\C^n\times\C^m) 
\end{equation} 
where $p_\mu(x,y;q,t) =\prod_{i=1}^{\ell(\mu)}p_{\mu_i}(x,y;q,t)$. 

From \cite[Theorem 5.4]{SV09a}, we recall that $\varphi_{n,m;q,t}$ intertwines the operator $\cM_{q,t}$ and the deformed Macdonald-Ruijsenaars operator $\cM_{n,m;q,t}$ defined by \eqref{cMnm1}--\eqref{cMnm2}:
\begin{equation}
\label{intertw}
\varphi_{n,m;q,t}\circ \cM_{q,t} = \cM_{n,m;q,t}\circ  \varphi_{n,m;q,t}.
\end{equation}
Combining \eqref{PeigEq} with \eqref{intertw}, we immediately see that
\begin{equation}
\label{SPeigEq}
\cM_{n,m;q,t}SP_\lambda(x,y;q,t) = d_\lambda(q,t)SP_\lambda(x,y;q,t); 
\end{equation} 
cf.~\cite[Corollary 5.7]{SV09a}.
We note that \eqref{qinvtinv} implies that the coefficients $c_{\mu\nu}(q,t)$ in \eqref{SPdef1} are invariant under the transformation $(q,t)\to (q^{-1},t^{-1})$, and since the deformed Newton sums in \eqref{pnmr} also have this invariance, \eqref{SPdef2} implies 
\begin{equation} 
\label{inv} 
SP_\lambda(x,y;q^{-1},t^{-1}) =SP_\lambda(x,y;q,t). 
\end{equation} 
Thus, the super-Macdonald polynomials $SP_\lambda((x_1,\ldots,x_n),(y_1,\ldots,y_m);q,t)$ are also eigenfunctions of the deformed Macdonald-Ruijsenaars operator $\cM_{n,m;q^{-1},t^{-1}}$ with eigenvalue $d_\lambda(q^{-1},t^{-1})$. 

We also recall that $SP_\lambda(x,y;q,t)$ for $(x,y)\in\C^n\times\C^m$ is non-zero if and only if $\lambda$ belongs to the following set of partitions, 
\begin{equation} 
\label{Hnm} 
H_{n,m} := \{ \lambda = (\lambda_1,\lambda_2,\ldots)\in\cP \ | \ \lambda_{n+1} \leq m \} ; 
\end{equation} 
cf.~\cite[Theorem 5.6]{SV09a}. 

Below we give an explicit representation of the super-Macdonald polynomials needed in the proof of our main result (this is a slight refinement of a result in \cite{SV09a}). 

\begin{lemma}
\label{lem:SPExp}
For $(x,y)\in\C^n\times\C^m$ and $\lambda\in H_{n,m}$, we have
\begin{equation}
\label{SPreps}
SP_\lambda(x,y;q,t) = \sum_\mu (-q^{-1/2}t^{1/2})^{|\mu|} P_{\lambda/\mu^\prime}(x;q,t)Q_\mu(y;t,q)
\end{equation}
where the sum runs over all partitions $\mu$ such that
\begin{equation}
\label{restr} 
(\langle \lambda^\prime_1-n\rangle,\ldots,\langle \lambda^\prime_m-n\rangle)\subseteq \mu\subseteq (\lambda^\prime_1,\ldots,\lambda^\prime_m)  
\end{equation} 
where $\langle k\rangle:=\max(k,0)$. 
\end{lemma}

\begin{proof}
Working with infinite sets of variables $x=(x_1,x_2,\ldots)$ and $y=(y_1,y_2,\ldots)$, we infer from \eqref{PPQ} that  
$$
P_\lambda(x,y;q,t) = \sum_{\mu\subseteq\lambda^\prime}P_{\lambda/\mu^\prime}(x;q,t)P_{\mu^\prime}(y;q,t)
$$
(where we have taken $\mu\to\mu^\prime$ and used that $\mu^\prime\subseteq\lambda$ if and only if $\mu\subseteq\lambda^{\prime}$). 
Applying, with respect to $y$, the automorphism $\sigma_{q,t}: \Lambda\to\Lambda$ characterised by
$$
(\sigma_{q,t}(p_r))(y) = (\omega_{q,t}(p_r))(-q^{-1/2}t^{1/2}y) = -\frac{q^{r/2}-q^{-r/2}}{t^{r/2}-t^{-r/2}}p_r(y)\ \ \ (r\geq 1),
$$
and setting $x_i=0$ for $i>n$ and $y_j=0$ for $j>m$, it is clear from \eqref{pnmr}--\eqref{SPdef2} that $P_\lambda(x,y;q,t)$ is mapped to $SP_\lambda((x_1,\ldots,x_n),(y_1,\ldots,y_m);q,t)$, so that \eqref{duality} and \eqref{homQ} imply 
\begin{multline*}
SP_\lambda((x_1,\ldots,x_n),(y_1,\ldots,y_m);q,t)\\
= \sum_{\mu\subseteq\lambda^\prime}(-q^{-1/2}t^{1/2})^{|\mu|} P_{\lambda/\mu^\prime}((x_1,\ldots,x_n);q,t)Q_\mu((y_1,\ldots,y_m);t,q).
\end{multline*}
To justify the conditions in \eqref{restr}: (i) Note that $\mu\not\subseteq (\lambda^\prime_1,\ldots,\lambda^\prime_m)$ and $\mu\subseteq\lambda^\prime$ can be simultaneously satisfied only if $\mu_{m+1}\neq 0$, in which case $Q_\mu(y_1,\ldots,y_m;t,q)\equiv 0$, (ii) $(\langle \lambda^\prime_1-n\rangle,\ldots,\langle \lambda^\prime_m-n\rangle)\not\subseteq \mu$ is only possible if there exists $j=1,\ldots,m$ such that $\lambda^\prime_j-\mu_j>n$, but then $P_{\lambda/\mu^\prime}((x_1,\ldots,x_n);q,t)\equiv 0$ by Lemma~\ref{lem:f} (a detailed justification of the latter can be found in Appendix~\ref{app:detail}).  
\end{proof}

\section{Results}
\label{sec:results}
We now turn to our results. In Subsection \ref{sec:orth}, we introduce the relevant scalar product on the space $\Lambda_{n,m;q,t}$, spanned by the super-Macdonald polynomials, and state our main results in Theorem~\ref{thm}. The proof of this theorem is deferred to Section~\ref{sec:proof}. The Hilbert space interpretation of deformed Macdonald-Ruijsenaars operators and super-Macdonald polynomials, as provided by this scalar product, is discussed in Section~\ref{sec:hil}. 

In what follows, we use the short-hand notation
\begin{equation*} 
\label{domN} 
\dd\omega_n(x):=\frac{1}{(2\pi\ii)^n}\frac{\dd x_1}{x_1}\cdots \frac{\dd x_n}{x_n}
\end{equation*} 
for variables $x=(x_1,\ldots,x_n)\in\C^n$ and $n\in\Z_{\geq 1}$; we also recall the definition of the $n$-torus $\T_\xi^n$ of radius $\xi>0$ in \eqref{tori}.

\subsection{Orthogonality}
\label{sec:orth}
We let $L_{n,m}=\mathbb{C}[x_1^{\pm 1},\ldots,x_n^{\pm 1},y_1^{\pm 1},\ldots,y_m^{\pm 1}]$ be the algebra of complex Laurent polynomials in the variables $x_1,\ldots,x_n$ and $y_1,\ldots,y_m$. For $f\in L_{n,m}$, we define its conjugate $f^*$ by
\begin{equation}\label{conj}
f^*(x,y)=\overline{f(\bar{x}^{-1},\overline{y}^{-1})},
\end{equation}
where $\bar{x}^{-1}$ and $\bar{y}^{-1}$ are as in \eqref{barinv}. We recall that $\Lambda_{n,m;q,t}$ is the space of polynomials $P(x,y)$ in the variables $(x,y)=((x_1,\ldots,x_n),(y_1,\ldots,y_m))\in\C^n\times\C^m$ with complex coefficients satisfying the conditions in \eqref{Lamnm}--\eqref{qinv}.

As already described in the introduction, the Hermitian product of $P,Q\in\Lambda_{n,m;q,t}$ is obtained by integrating the product of $P(x,y)Q^*(x,y)$  with the weight function $\Delta_{n,m}(x,y;q,t)$ in \eqref{Deln}--\eqref{wghtFunc} over the $n+m$-dimensional torus $\T_\xi^n\times \T_{\xi^\prime}^m$ with suitable radii $\xi,\xi^\prime>0$; see \eqref{prod0S}. To see that we need to restrict the radii, we note that, while $P(x,y)Q^*(x,y)\in L_{n,m}$, and thus is holomorphic for $(x,y)\in(\mathbb{C}^*)^n\times(\mathbb{C}^*)^m$, the weight function $\Delta_{n,m}(x,y;q,t)$ is meromorphic with simple poles located along the hyperplanes
\begin{subequations} 
\label{poles} 
\begin{equation}
\label{xypls}
x_i = q^{\frac{\delta}{2}}t^{-\frac{\delta}{2}}y_j\ \ \ (i = 1,\ldots,n,\ j = 1,\ldots,m,\ \delta = \pm 1), 
\end{equation}
\begin{equation}
\label{xpls}
tq^kx_i = x_{i^\prime}\ \ \ (1\leq i\neq i^\prime\leq n,\ \ k\in\mathbb{Z}_{\geq 0}), 
\end{equation}
\begin{equation}
\label{ypls}
qt^ky_j = y_{j^\prime}\ \ \ (1\leq j\neq j^\prime\leq m,\ \ k\in\mathbb{Z}_{\geq 0}).
\end{equation}
\end{subequations} 
Clearly, $\T_\xi^n\times \T_{\xi^\prime}^m$ is contained in the complement of these hyperplanes provided the radii $\xi,\xi^\prime>0$ are constrained as follows: 
\begin{equation} 
\label{xixip} 
\xi/\xi^\prime < \min_{\delta=\pm 1}\big(q^{\frac{\delta}{2}}t^{-\frac{\delta}{2}}\big) \ \ \text{or}\ \ \xi/\xi^\prime > \max_{\delta=\pm 1}\big(q^{\frac{\delta}{2}}t^{-\frac{\delta}{2}}\big); 
\end{equation} 
if we restrict ourselves to such radii, we avoid all singularities of the integrand and thus obtain well-defined integrals; see Remark~\ref{remarkxi}.
Note that the condition in \eqref{xixip} can be written in a more compact way as follows, $|\log(\xi/\xi^\prime)|>\frac12|\log(q/t)|$.

\begin{definition}
\label{def:prod} 
For $\xi,\xi^\prime>0$ satisfying either of the two conditions in \eqref{xixip}, 
we define a sesquilinear form $\langle\cdot,\cdot\rangle_{n,m;q,t}^\prime$ on $\Lambda_{n,m;q,t}$ by
\begin{equation}
\label{prod}
\langle P,Q\rangle_{n,m;q,t}^\prime = \frac{1}{n!m!}\int_{\T_\xi^n}\dd\omega_n(x)\int_{\T_{\xi^\prime}^m}\dd\omega_m(y)\Delta_{n,m}(x,y;q,t)P(x,y)Q^*(x,y)
\end{equation}
for arbitrary $P,Q\in \Lambda_{n,m;q,t}$.
\end{definition}

Using that the integrand in \eqref{prod} is analytic everywhere except along the hyperplanes \eqref{poles}, it is not difficult to prove that this sequilinar form does not depend on $\xi,\xi^\prime$ as long as they vary over only one of the two regions in \eqref{xixip}; see Lemma~\ref{lem:indep}. This argument applies to any Laurent polynomials $P,Q\in L_{n,m}$, but it does not rule out the possibility that the value of $\langle P,Q\rangle_{n,m;q,t}^\prime$ in the former region $\xi/\xi'< \min_{\delta=\pm 1}\big(q^{\frac{\delta}{2}}t^{-\frac{\delta}{2}}\big)$ is different from that in the latter region $\xi/\xi'> \max_{\delta=\pm 1}\big(q^{\frac{\delta}{2}}t^{-\frac{\delta}{2}}\big)$. However, as we will show, if $P$ and $Q$ belong to $\Lambda_{n,m;q,t}$, then the value of $\langle P,Q\rangle_{n,m;q,t}^\prime$ is the same in both regions.

In order to appreciate the significance of the conditions  \eqref{Lamnm}--\eqref{qinv}, it is instructive to consider the simplest non-trivial case $n=m=1$, in which the above claim can be verified by direct computations; the interested reader can find the details in Appendix~\ref{app:nm1}. 

To state our main result in Theorem~\ref{thm} below, we need two mappings $e$ and $s$ on partitions. 
For that,  we observe that a partition $\lambda\in H_{n,m}$ such that $(m^n)\subseteq\lambda$ satisfies the conditions 
\begin{equation*} 
\lambda_n\geq m\geq \lambda_{n+1}
\end{equation*} 
and, for this reason, it can be written as 
\begin{equation*} 
\lambda=(\enm{\lambda}+(m^n),\snm{\lambda}')
\end{equation*} 
 with two partitions $\enm{\lambda}$ and $\snm{\lambda}$ of lengths less or equal to $n$ and $m$, respectively, 
and determined by $\lambda=(\lambda_1,\lambda_2,\ldots)$ as follows,  
\begin{equation} 
\label{es}
\begin{split} 
\enm{\lambda}&:=(\lambda_1-m,\lambda_2-m,\ldots,\lambda_n-m) = (\lambda_{m+1}^\prime,\lambda_{m+2}^\prime,\ldots)^\prime\\
\snm{\lambda}&:=(\lambda_{n+1},\lambda_{n+2},\ldots)^\prime = (\lambda_1^\prime-n,\ldots,\lambda_m^\prime-n); 
\end{split} 
\end{equation} 
see \cite[Section~2.2]{AHL19} for more details on these mappings $e$ (short for east) and $s$ (short for south), including the motivation for these names. 
To simplify notation, we write $e(\lambda)$ short for $\enm{\lambda}$ and $s(\lambda)$ short for $\snm{\lambda}$ if no confusion can arise. 

\begin{theorem} 
\label{thm}
(a) The sesquilinear form $\langle\cdot,\cdot\rangle_{n,m;q,t}^\prime$ from Definition \ref{def:prod} is Hermitian, i.e.
$$
\langle P,Q\rangle_{n,m;q,t}^\prime = \overline{\langle Q,P\rangle_{n,m;q,t}^\prime},\ \ \ (P,Q\in \Lambda_{n,m;q,t}),
$$
and independent of the integration radii $\xi,\xi^\prime>0$ provided \eqref{xixip} holds true.

(b) The super-Macdonald polynomials 
$SP_\lambda=SP_\lambda((x_1,\ldots,x_n),(y_1,\ldots,y_m);q,t)$, $\lambda\in H_{n,m}$,  satisfy the orthogonality relations
\begin{equation}
\label{orth}
\langle SP_\lambda,SP_\mu\rangle^\prime_{n,m;q,t} = \delta_{\lambda\mu}N_{n,m}(\lambda;q,t),
\end{equation}
with (quadratic) norms
\begin{equation}
\label{zNnm}
N_{n,m}(\lambda;q,t) = 0\quad \text{if}\ \  (m^n)\not\subseteq\lambda
\end{equation}
and
\begin{equation}
\label{nzNnm}
\begin{split}
N_{n,m}(\lambda;q,t) &= (t/q)^{|s(\lambda)|}\frac{b_{e(\lambda)}(q,t)b_{s(\lambda)}(t,q)}{b_\lambda(q,t)}\\
&\quad \cdot N_n(e(\lambda);q,t)N_{m}(s(\lambda);t,q)\ \ \text{if}\ \ (m^n)\subseteq\lambda, 
\end{split}
\end{equation}
cf.~\eqref{blam}, \eqref{Nn} and \eqref{es}. 
\end{theorem}

\subsection{Hilbert space interpretation}
\label{sec:hil}

From Theorem \ref{thm}, we see that the kernel of the Hermitian product \eqref{prod} is spanned by the super-Macdonald polynomials with zero norm:
\begin{equation*}
\begin{split}
K_{n,m;q,t} &:= \ker\, \langle\cdot,\cdot\rangle_{n,m;q,t}^\prime\\
&= \spn\big\{SP_\lambda((x_1,\ldots,x_n),(y_1,\ldots,y_m);q,t)\mid (m^n)\not\subseteq\lambda\in H_{n,m}\big\}.
\end{split}
\end{equation*}
Since the remaining norms $N_{n,m}(\lambda;q,t)$, where $(m^n)\subseteq\lambda$, are positive, we have the following result.

\begin{proposition}
The Hermitian form $\langle\cdot,\cdot\rangle_{n,m;q,t}^\prime$ descends to a (positive definite) scalar product on the factor space
$$
V_{n,m;q,t} := \Lambda_{n,m;q,t}/K_{n,m;q,t},
$$
and the renormalised super-Macdonald polynomials
$$
\widetilde{SP}_\lambda(x,y;q,t) := N_{n,m}(\lambda;q,t)^{-1/2}SP_\lambda(x,y;q,t),\ \ \ (m^n)\subseteq\lambda,
$$
with $N_{n,m}(\lambda;q,t)$ given by \eqref{nzNnm}, yield an orthonormal basis in the resulting (pre-)Hilbert space.
\end{proposition}

Moreover, since the deformed Macdonald-Ruijsenaars operators $\cM_{n,m;q,t}$ and $\cM_{n,m;q^{-1},t^{-1}}$ leave $K_{n,m;q,t}$ invariant and their eigenvalues are all real (cf.~\eqref{SPeigEq}--\eqref{inv}), they define (essentially) self-adjoint operators in $V_{n,m;q,t}$. Hence, we have assembled everything needed for a quantum mechanical interpretation of the model defined by $\cM_{n,m;q,t}$ and $\cM_{n,m;q^{-1},t^{-1}}$.

It is interesting to note that the subset $\{\lambda\in H_{n,m}\mid (m^n)\subset \lambda\}\subset H_{n,m}$, which labels super-Macdonald polynomials with nonzero norm, is in a simple one-to-one correspondence with the subset
$\{ (\mu,\nu)\in\cP\times\cP\mid \ell(\mu)\leq n,\ \ell(\nu)\leq m\}\subset \cP\times\cP$, given explicitly by
\begin{equation*} 
\lambda = ((m^n)+\mu,\nu') ,\ \ \  \mu=e(n,m;\lambda),\ \ \nu=s(n,m;\lambda),
\end{equation*} 
cf.~\eqref{es}. 
Moreover, while the physical interpretation of a partition $\lambda$ of arbitrary length is not clear, the partitions $\mu=(\mu_1,\ldots,\mu_n)$ and $\nu=(\nu_1,\ldots,\nu_m)$ have a natural physical interpretation as momentum quantum numbers:
The corresponding super-Macdonald polynomials
\begin{equation*} 
SP_{(\mu+(m^n),\nu')}(x,y;q,t) \ \ \ ((x,y)\in\C^n\times\C^m)
\end{equation*} 
yield an orthogonal basis in $V_{n,m;q,t}$, and it is natural to interpret them as wave functions describing $n$ and $m$ particles of two different kinds labeled by a pair $(\mu,\nu)$ of momentum quantum numbers.

From a physics point of view, it would be natural to express wave functions and operators in terms of the ``additive'' variables $(u,v)=(u_1,\ldots,u_n,v_1,\ldots,v_m)$ defined as follows,
\begin{equation}
\label{uv}
x_i = \ee^{2\pi\ii u_i/L},\quad y_j=\ee^{2\pi\ii v_j/L}\ \ \ (L>0) 
\end{equation} 
and parameters
\begin{equation}
\label{betagamma}
q = \ee^{-2\pi\beta/L},\quad t = \ee^{-2\pi\gamma/L}\ \ \ (\beta,\gamma>0),
\end{equation}
cf.~\cite{RS86,Rui87}.
Here $u_i\in[-L/2,L/2]$ and $v_i\in [-L/2,L/2]$ have the physical interpretation of position coordinates of particles. Indeed, in the original quantum field theoretic context of the Ruijsenaars model one is, eventually, interested in the limit $L\to\infty$, where space is the real line, but, to have a well-defined model, it is convenient to work with a circle of finite circumference $L$. 

Taking, for simplicity, parameters $r,a,b>0$, we recall from \cite[Section III.C \& Section V.C]{Rui97} the trigonometric Gamma function
$$
G(r,a;z) = \prod_{k=0}^\infty \big(1-\exp(2\ii r(z+\ii a k+\ii a/2))\big)^{-1}
$$
and weight function
\begin{equation}
\label{w}
w(r,a,b;z) = \frac{G(r,a;z+\ii b-\ii a/2)G(r,a;-z+\ii b-\ii a/2)}{G(r,a;z-\ii a/2)G(r,a;-z-\ii a/2)}.
\end{equation}
We note that $w(z)=w(r,a,b;z)$ is a (globally) meromorphic function with simple poles located at
$$
z = j\pi/r-\ii b-\ii a(2k-1)\ \ \ (j\in\mathbb{Z},\ k\in\mathbb{Z}_{\geq 1}),
$$
and zeros at
$$
z = j\pi/r-\ii a(2k-1)\ \ \ (j\in\mathbb{Z},\ k\in\mathbb{Z}_{\geq 1}).
$$
Due to the manifest complex conjugation property
\begin{equation*}
\overline{G(r,a;z)}=G(r,a;-\bar{z}),
\end{equation*}
it follows, in particular, that $w(z)$ is a regular and (strictly) positive function in $\mathbb{R}$. Assuming that $u\in(\mathbb{R}+\ii\epsilon)^n$ and $v\in(\mathbb{R}+\ii\epsilon^\prime)^m$ for some $\epsilon,\epsilon^\prime\in\R$, we can thus introduce the (formal) groundstate wave function
\begin{multline}
\label{Psi0}
\Psi_0(u,v;\beta,\gamma)\\
= \frac{\left(\prod_{1\leq i<j\leq n}w(\pi/L,\beta,\gamma;u_i-u_j)\right)^{1/2} \left(\prod_{1\leq i<j\leq m}w(\pi/L,\gamma,\beta;v_i-v_j)\right)^{1/2}}{\prod_{i=1}^n\prod_{j=1}^m 2\sin(\pi(u_i-v_j+\ii\gamma/2-\ii\beta/2)/L)},
\end{multline}
(where we take the positive square roots), and obtain a natural factorisation of the weight function $\Delta_{n,m}$ \eqref{wghtFunc}, as detailed in the following Lemma.

\begin{lemma}
\label{Lemma:fact}
For $(x,y)\in \T_\xi^n\times \T_{\xi^\prime}^m$, we have
\begin{equation*}
\Delta_{n,m}(x,y;q,t) = \exp(nm\pi(\gamma-\beta)/L)\Psi_0(u,v;\beta,\gamma)\overline{\Psi_0}(u,v;\beta,\gamma),
\end{equation*}
where $\overline{\Psi_0}(u,v;\beta,\gamma):=\overline{\Psi_0(\bar{u},\bar{v};\beta,\gamma)}$.
\end{lemma}

\begin{proof}
For $x$ and $z$ complex variables related as
$$
x=\ee^{2\pi \ii z/L},
$$
we use \eqref{betagamma} to deduce
\begin{equation*}
\begin{split}
\frac{(x;q)_\infty}{(tx;q)_\infty} &= \prod_{k=0}^\infty \frac{1-\exp\big(\frac{2\pi \ii}{L}(z+\ii\beta k)\big)}{1-\exp\big(\frac{2\pi \ii}{L}(z+\ii\gamma+\ii\beta k)\big)}\\
&= \frac{G(\pi/L,\beta;z+\ii\gamma-\ii\beta/2)}{G(\pi/L,\beta;z-\ii\beta/2)}
\end{split}
\end{equation*}
and
\begin{multline*}
(1-q^{-1/2}t^{1/2}x)(1-q^{-1/2}t^{1/2}x^{-1})\\
= 4\ee^{\pi(\beta-\gamma)/L}\sin(\pi(z+\ii\gamma/2-\ii\beta/2)/L)\sin(\pi(z-\ii\gamma/2+\ii\beta/2)/L).
\end{multline*}
From \eqref{wghtFunc} and \eqref{w}--\eqref{Psi0}, the statement can now be inferred by a straightforward computation.
\end{proof}

If we consider wave functions of the form
\begin{equation*}  
\Psi^{(P)}(u,v;\beta,\gamma)  : = P(x,y)\Psi_0(u,v;\beta,\gamma) \quad (u=(u_1,\ldots,u_n),\, v=(v_1,\ldots,v_m))
\end{equation*} 
with $P\in\Lambda_{n,m;q,t}$, then we can use Lemma \ref{Lemma:fact} to rewrite our Hermitian form as a suitably regularised version of a conventional Hilbert space product for a quantum mechanical model describing particles moving on the circle $[-L/2,L/2]$. 

\begin{proposition}
Let $\epsilon,\epsilon^\prime\in\mathbb{R}$ be such that $|\epsilon-\epsilon^\prime|>|\gamma-\beta|/2$. Then we have
\begin{multline}
\label{altProd}
\int_{([-L/2,L/2]+\ii\epsilon)^n}du\int_{([-L/2,L/2]+\ii\epsilon^\prime)^m}dv\, \Psi^{(P)}(u,v;\beta,\gamma)\overline{\Psi^{(Q)}}(u,v;\beta,\gamma)\\
= n!m! L^{n+m}\exp(nm\pi(\beta-\gamma)/L)\langle P,Q\rangle_{n,m;q,t}^\prime
\end{multline}
for all $P,Q\in\Lambda_{n,m;q,t}$, where $\overline{\Psi^{(Q)}}(u,v;\beta,\gamma)=\overline{\Psi^{(Q)}(\bar{u},\bar{v};\beta,\gamma)}$.
\end{proposition}

\begin{proof}
Taking
$$
\xi = \ee^{-2\pi\epsilon/L},\ \ \ \xi^\prime = \ee^{-2\pi\epsilon^\prime/L},
$$
it is readily seen that the constraint \eqref{xixip} is equivalent to $|\epsilon-\epsilon^\prime|>|\gamma-\beta|/2$. Changing variables according to \eqref{uv}--\eqref{betagamma} in \eqref{prod} and invoking Lemma \ref{Lemma:fact}, the equality \eqref{altProd} results.
\end{proof}

\begin{remark} 
Note that, from a physics point of view, the positions $u_i$ and $v_j$ are real, but one has to continue the arguments of the super-Macdonald polynomials to the complex plane in order to compute their scalar product. This bears some resemblance to the fact that an eigenfunction of Ruijsenaars' (analytic) difference operators needs to have sufficient analyticity in order for the corresponding eigenvalue equations to make sense, see e.g.~\cite{Rui01}.
\end{remark}

As we demonstrate in Appendix \ref{app:Id}, the terms in the deformed Macdonald--Ruijsenaars operator $\cM_{n,m;q,t}$ in \eqref{cMnm1}--\eqref{cMnm2} not involving a shift operator add upp to a constant. Dropping this overall constant, we get the operator
\begin{equation}
\label{Mnm} 
M_{n,m;q,t} :=  \frac{t^{1-n}}{1-q} \sum_{i=1}^n A_i(x,y;q,t)T_{q,x_i}
+ \frac{q^{m-1}}{1-t^{-1}} \sum_{j=1}^m B_j(x,y;q,t)T_{t^{-1},y_j}.
\end{equation}
Changing variables and parameters according to \eqref{uv} and performing a similarity transformation with $\Psi_0$, a direct computation, using the difference equation
$$
\frac{G(z+\ii a/2)}{G(z-\ii a/2)}=1-\ee^{2\ii rz},
$$
satisfied by the trigonometric Gamma function (cf.~\cite[Section III.C]{Rui97}), yields
\begin{multline*}
R_{n,m;\beta,\gamma}^- := \ee^{(m-1)\pi\beta/L}\ee^{(1-n)\pi\gamma/L}\Psi_{0} M_{n,m;q,t} \Psi_{0}^{-1}\\
= \frac{1}{2\sinh(\pi\beta/L)}\sum_{i=1}^n \prod_{i^\prime\neq i}^n f_+(u_i-u_{i^\prime};\gamma)\cdot \ee^{\ii\beta\partial_{u_i}}\cdot \prod_{i^\prime\neq i}^n f_-(u_i-u_{i^\prime};\gamma)\\
 -\frac{1}{2\sinh(\pi\gamma/L)}\sum_{j=1}^m \prod_{i=1}^n \frac{\sin(\pi(u_i-v_j-\ii\gamma/2+\ii\beta/2)/L)}{\sin(\pi(u_i-v_j+\ii\gamma/2+\ii\beta/2)/L)}\\
 \cdot \prod_{j^\prime\neq j}^m f_-(v_j-v_{j^\prime};\beta)\cdot \ee^{-\ii\gamma\partial_{v_j}}
 \cdot \prod_{j^\prime\neq j}^m f_+(v_j-v_{j^\prime};\beta)\\
 \cdot \prod_{i=1}^n \frac{\sin(\pi(u_i-v_j+\ii\gamma/2-\ii\beta/2)/L)}{\sin(\pi(u_i-v_j-\ii\gamma/2-\ii\beta/2)/L)},
\end{multline*}
with
$$
f_\pm(z;\gamma) = \left(\frac{\sin(\pi(z\pm \ii\gamma)/L)}{\sin(\pi z/L)}\right)^{1/2},
$$
and where $\partial_{u_i}=\partial/\partial u_i$ and $\partial_{v_j}=\partial/\partial v_j$.

The structure of these operators occupies a sort of middle-ground between the trigonometric degeneration of Ruijsenaars' original (undeformed elliptic) operator ${\hat S}_{-1}$ and a similarity transform $\mathcal{A}_{-1}=U^{-1/2}{\hat S}_{-1}U^{1/2}$ with a trigonometric `scattering function' $U$. (Explicit expressions for the latter operator and the pertinent scattering function (in the hyperbolic case) can, e.g., be found in~\cite{HR14}.) In particular, when $m=0$ we recover the trigonometric instance of ${\hat S}_{-1}$.

Moreover, as shown in Appendix~\ref{app:relativistic}, the operators $M_{n,m;q,t}$ and $M_{n,m;q^{-1},t^{-1}}$, or equivalently $R_{n,m;\beta,\gamma}^-$ and $R_{n,m;\beta,\gamma}^+:=\ee^{(1-m)\pi\beta/L}\ee^{(n-1)\pi\gamma/L}\Psi_{0} M_{n,m;q^{-1},t^{-1}} \Psi_{0}^{-1}$, define a model that is relativistically invariant, for arbitrary particle numbers $n$ and $m$.

\section{Proofs}
\label{sec:proof}
This section is devoted to the proof of Theorem~\ref{thm}. In place of \eqref{prod}, we write
\label{Sec:Proofs}
\begin{multline}
\label{prod1}
\langle P,Q\rangle_{n,m;q,t}^\prime(\xi,\xi^\prime)\\
= \frac{1}{n!m!}\int_{\T_\xi^n}d\omega_n(x)\int_{\T_{\xi^\prime}^m}d\omega_m(y)\Delta_{n,m}(x,y;q,t)P(x,y)Q^*(x,y)
\end{multline}
for $P,Q\in \Lambda_{n,m;q,t}$, so that we easily can keep track of the choice of integration radii $\xi,\xi^\prime>0$.
Introducing the maximum function
$$
M(q,t) := \max_{\delta=\pm 1}\big(q^{\frac{\delta}{2}}t^{-\frac{\delta}{2}}\big),
$$
we note that the conditions \eqref{xixip} can be expressed as
\begin{equation}
\label{xixipAlt}
\xi/\xi^\prime > M(q,t) \ \ \text{or}\ \ \xi/\xi^\prime < 1/M(q,t).
\end{equation}

In Section~\ref{sec:proof1}, we prove preliminary results on the sesquilinear form given by \eqref{prod1}.
Based on this, we prove Theorem~\ref{thm} in Section~\ref{sec:proof2}.  

\subsection{Preliminary results}
\label{sec:proof1}
First, we establish a simple transformation property of the form \eqref{prod1} under the inversion $(\xi,\xi^\prime)\to(1/\xi,1/\xi^\prime)$ of integration radii.
 
\begin{lemma}
\label{lem:inv}  
For all $P,Q\in\Lambda_{n,m;q,t}$, we have
\begin{equation}
\label{prodExpr}
\langle P,Q\rangle_{n,m;q,t}^\prime(\xi,\xi^\prime) = \overline{ \langle Q,P\rangle_{n,m;q,t}^\prime(1/\xi,1/\xi^\prime) }. 
\end{equation} 
\end{lemma}

\begin{proof}
We find it convenient to work with the ``additive'' variables given by \eqref{uv} with $L=2\pi$. In order to avoid cumbersome and lengthy formulas, we suppress the parameters $q$, $t$ and use the short-hand notation
$$
\exp(\ii u) := (\ee^{\ii u_1},\ldots,\ee^{\ii u_n}),\ \ \ \exp(\ii v) := (\ee^{\ii v_1},\ldots,\ee^{\ii v_m}).
$$
Observing that
$$
\Delta_{n,m}(x,y) = \Delta_{n,m}(x^{-1},y^{-1}),
$$
we thus rewrite the left-hand side of \eqref{prodExpr} as
\begin{multline*}
\frac{1}{(2\pi)^{n+m}}\int_{[-\pi,\pi]^n}du \int_{[-\pi,\pi]^m}dv\, \overline{\Delta_{n,m}(\xi^{-1}\exp(\ii u),(\xi^\prime)^{-1}\exp(\ii v))}\\
\cdot P(\xi\exp(\ii u),\xi^\prime\exp(\ii v))\overline{Q(\xi^{-1}\exp(\ii u),(\xi^\prime)^{-1}\exp(\ii v))}.
\end{multline*}
Finally, using the observation
$$
P(\xi\exp(\ii u),\xi^\prime\exp(\ii v)) = \overline{P^*(\xi^{-1}\exp(\ii u),(\xi^\prime)^{-1}\exp(\ii v))},
$$
we see that this integral is equal to the right-hand side of \eqref{prodExpr}.
\end{proof} 

We proceed to show that \eqref{prod1} is invariant under continuous deformations of the integration radii as long as they satisfy \eqref{xixipAlt}.

\begin{lemma}
\label{lem:indep}
For any $P,Q\in\Lambda_{n,m;q,t}$, the value of $\langle P,Q\rangle^\prime_{n,m;q,t}(\xi,\xi^\prime)$ does not change as $\xi,\xi^\prime>0$ vary within one of the two regions  $\xi/\xi^\prime>M(q,t)$ and $\xi/\xi^\prime<1/M(q,t)$.
\end{lemma}

\begin{proof}
We note that $\xi/\xi^\prime<1/M(q,t)$ if and only if $(1/\xi)/(1/\xi^\prime)>M(q,t)$. Hence, thanks to Lemma \ref{lem:inv}, we may and shall restrict attention to the region $\xi/\xi^\prime>M(q,t)$.

By Cauchy's theorem, we can deform the integration contours in \eqref{prod1} one at a time, without changing the value of the integral, as long as we do not encounter any of the poles \eqref{xypls}--\eqref{ypls}. In particular, taking
$$
\boldsymbol{\xi} = (\xi_1,\ldots,\xi_n)\in (0,\infty)^n,\ \ \ \boldsymbol{\xi}^\prime = (\xi^\prime_1,\ldots,\xi^\prime_m)\in (0,\infty)^m,
$$
such that
\begin{subequations} 
\begin{equation} 
\label{xixipReg1}
t^{1/2}<\xi_i/\xi_{i^\prime}<t^{-1/2}\ \ (1\leq i\neq i^\prime\leq n),
\end{equation}
\begin{equation}
\label{xixipReg2}
q^{1/2}<\xi_j^\prime/\xi_{j^\prime}^\prime<q^{-1/2}\ \ (1\leq j\neq j^\prime\leq n), 
\end{equation}
\begin{equation}
\label{xixipReg3}
\xi_i/\xi_j^\prime>M(q,t)\ \ (i = 1,\ldots,n,\ j = 1,\ldots,m),
\end{equation}
\end{subequations}
we can replace $\T_\xi^n$ and $\T_{\xi^\prime}^m$ by
$$
\T_{\boldsymbol{\xi}}^n := \T_{\xi_1}\times\cdots\times \T_{\xi_n},\ \ \ \T_{\boldsymbol{\xi}^\prime}^m := \T_{\xi^\prime_1}\times\cdots\times \T_{\xi^\prime_m},
$$
respectively, and the resulting integral
$$
\int_{\T_{\boldsymbol{\xi}}^n}d\omega_n(x)\int_{\T_{\boldsymbol{\xi}^\prime}^m}d\omega_m(y)\, \Delta_{n,m}(x,y;q,t)P(x,y)\overline{Q}(x^{-1},y^{-1})
$$
is independent of $\boldsymbol{\xi}\in (0,\infty)^n$ and $\boldsymbol{\xi}^\prime\in (0,\infty)^m$ provided 
the inequalities \eqref{xixipReg1}--\eqref{xixipReg3} are satisfied. Indeed, these inequalities clearly define an open and (pathwise) connected subset of $(0,\infty)^n\times(0,\infty)^m$, so that, regardless of the initial integration radii, any admissible radii $\boldsymbol{\xi}$ and $\boldsymbol{\xi}^\prime$ can be reached in a finite number of steps, where each step consist of a deformation of a single radius. Since \eqref{xixipReg1}--\eqref{xixipReg3} are satisfied whenever $\xi_1=\cdots=\xi_n=\xi$ and $\xi^\prime_1=\cdots=\xi_m^\prime=\xi^\prime$ with $\xi/\xi^\prime>M(q,t)$, the lemma follows.
\end{proof} 

We note that this result leaves open the possibility that the sesquilinear form \eqref{prod1} takes different values in the two regions $\xi/\xi^\prime>M(q,t)$ and $\xi/\xi^\prime<1/M(q,t)$ --- this possibility is only ruled out by arguments given in Section~\ref{sec:proof2}.

To conclude, we show that the deformed Macdonald operator $\cM_{n,m;q,t}$ (as defined by \eqref{cMnm1}--\eqref{cMnm2}) is self-adjoint.

\begin{lemma}
\label{lem:selfAdj}
For all $P,Q\in\Lambda_{n,m;q,t}$, we have
\begin{equation}
\label{selfAdj}
\langle \cM_{n,m;q,t}P,Q\rangle_{n,m;q,t}^\prime(\xi,\xi^\prime) = \langle P,\cM_{n,m;q,t}Q\rangle_{n,m;q,t}^\prime(\xi,\xi^\prime) 
\end{equation}
provided $\xi,\xi'>0$ satisfy \eqref{xixipAlt}. 
\end{lemma}

\begin{proof}
We observe that the terms in $\cM_{n,m;q,t}$ \eqref{cMnm1}--\eqref{cMnm2} not involving shift operators add up to a real constant,  
which clearly is self-adjoint:
\begin{equation}
\label{Id}  
\frac{t^{1-n}}{1-q}\sum_{i=1}^nA_i(x,y;q,t) + \frac{q^{m-1}}{1-t^{-1}}B_j(x,y;q,t) = \frac{1-t^{-n}q^m}{(1-t^{-1})(1-q)} ; 
\end{equation} 
for the convenience of the reader, we include an elementary proof of this identity in Appendix~\ref{app:Id}.

Since $\cM_{n,m;q,t}$ leaves $\Lambda_{n,m;q,t}$ invariant \cite[Proposition 5.3]{SV09a}, Lemma \ref{lem:indep} ensures that no generality is lost when replacing \eqref{xixipAlt} with
\begin{equation} 
\label{xixip1} 
\xi/\xi^\prime<(qt)^{1/2}\ \  \text{ or } \ \ \xi/\xi^\prime>(qt)^{-1/2}.
\end{equation} 
Under this stronger condition on the integration radii $\xi$ and $\xi^\prime$, we proceed to show that all terms $A_i(x,y;q,t)T_{q,x_i}$ and $B_j(x,y;q,t)T_{t^{-1},y_j}$, for $i=1,\ldots,n$ and $j=1,\ldots,m$, in $\cM_{n,m;q,t}$ are separately self-adjoint. In what follows, we do not indicate the dependence on $q$ and $t$, to simplify notation. 

Fixing $i=1,\ldots,n$, we introduce the function
\begin{equation}
\label{Wi}  
W_i(x,y):= \prod_{i'\neq i}^n \frac{(x_i/x_{i'};q)_\infty}{(tx_i/x_{i'};q)_\infty}\cdot \prod_{j=1}^m \frac{1}{1-q^{-1/2}t^{1/2}x_i/y_j}, 
\end{equation} 
so that $W_i(x,y)W_i^*(x,y)$ amounts to all $x_i$-dependent factors in $\Delta_{n,m}(x,y)$, and note that self-adjointness of $A_i(x,y)T_{q,x_i}$ follows once we show that
\begin{multline}
\label{ATId}
\oint_{|x_i|=\xi} \frac{dx_i}{2\pi\ii x_i} W_i(x,y)W_i^*(x,y) Q^*(x,y)A_i(x,y)T_{q,x_i}P(x,y)  \\
= \oint_{|x_i|=\xi} \frac{dx_i}{2\pi\ii x_i}  W_i(x,y)W_i^*(x,y) P(x,y) (A_i(x,y)T_{q,x_i}Q)^*(x,y).
\end{multline} 
To this end, we observe that
\begin{equation*} 
\frac{T_{q,x_i} W_i(x,y)}{W_i(x,y)} =  \prod_{i'\neq i}^n \frac{1-tx_i/x_{i'}}{1-x_i/x_{i'}}\cdot \prod_{j=1}^m \frac{(1-q^{-1/2}t^{1/2}x_i/y_j)}{(1-q^{1/2}t^{1/2}x_i/y_j)} = q^{-m}A_i(x,y), 
\end{equation*}
and using this, we can write \eqref{ATId} as follows,
\begin{multline*} 
\oint_{|x_i|=\xi} \frac{dx_i}{2\pi\ii x_i} W_i^*(x,y) Q^*(x,y) (T_{q,x_i}W_iP)(x,y) \\
= \oint_{|x_i|=\xi} \frac{dx_i}{2\pi\ii x_i}  W_i(x,y) P(x,y) (T_{q,x_i}W_iQ)^*(x,y).  
\end{multline*}
We now change variables $x_i\to q x_i$ in the latter integral to obtain 
\begin{multline} 
\label{A1}
\oint_{|x_i|=\xi/q} \frac{dx_i}{2\pi\ii x_i} (T_{q,x_i}W_iP)(x,y) W_i^*(x,y)Q^*(x,y) \\
= \oint_{|x_i|=\xi} \frac{dx_i}{2\pi\ii x_i} (T_{q,x_i}W_iP)(x,y) W_i^*(x,y)Q^*(x,y),
\end{multline}
where the equality holds true due to Cauchy's theorem, since the integrand, which is the same in both integrals, is an analytic function of $x_i$ in the region $\xi\leq |x_i|\leq \xi/q$ when
\begin{equation} 
\label{A2}
|x_{i'}|=\xi\quad (i'=1,\ldots,i-1,i+1,\ldots,n),\quad |y_j|=\xi^\prime\quad (j=1,\ldots,m)
\end{equation} 
and \eqref{xixip1} is satisfied. A proof of this analyticity property of the integrand can be found in Appendix~\ref{app:check}. We have thus established \eqref{ATId} and, as previously noted, self-adjointness of $A_i(x,y)T_{q,x_i}$ immediately follows.

A proof of self-adjointness of the terms $B_j(x,y)T_{t^{-1},y_j}$, under the condition \eqref{xixip1}, can be obtained in a similar manner, and the details are therefore omitted.\footnote{The argument proving the self-adjointness of $B_j(x,y)T_{t^{-1},y_j}$ can be obtained from the one for $A_i(x,y)T_{q,x_i}$ by swapping $(n,x,q,\xi)\leftrightarrow(m,y,t^{-1},\xi^\prime)$.} 
\end{proof}

\subsection{Proof of Theorem~\ref{thm}}
\label{sec:proof2} 
Making use of results from Section~\ref{sec:proof1}, we prove the two parts of the theorem in reverse order.

\subsubsection{Part (b)}
Using \eqref{SPeigEq} and Lemma~\ref{lem:selfAdj}, we deduce
\begin{equation*} 
\big(d_\lambda(q,t)-d_\mu(q,t)\big)\langle SP_\lambda,SP_\mu\rangle_{n,m;q,t}^\prime(\xi,\xi^\prime) = 0.
\end{equation*}
Assuming that $\lambda\neq \mu$, we see from \eqref{dlam} that
$$
p_{\lambda\mu}(q,t) := t^{\max(\ell(\lambda),\ell(\mu))-1}\big(d_\lambda(q,t)-d_\mu(q,t)\big)
$$
is a non-zero polynomial function in $q$ and $t$ for $0<q,t<1$. Introducing its zero set
$$
Z_{\lambda\mu} := \{(q,t)\in(0,1)^2\mid p_{\lambda\mu}(q,t)=0\},
$$
we can thus conclude that $\langle SP_\lambda,SP_\mu\rangle_{n,m;q,t}^\prime(\xi,\xi^\prime)$ must vanish for all $(q,t)\in (0,1)^2\setminus Z_{\lambda\mu}$. Since the Hermitian form $\langle\cdot,\cdot\rangle_{n,m;q,t}^\prime(\xi,\xi^\prime)$ depends continuously on $(q,t)$ and $(0,1)^2\setminus Z_{\lambda\mu}$ is a dense (open) subset of $(0,1)^2$, it clearly follows that the orthogonality relations \eqref{orth} hold true for all $(q,t)\in(0,1)^2$.

We proceed to compute the norms $N_{n,m}(\lambda;q,t):=\langle SP_\lambda,SP_\lambda\rangle_{n,m;q,t}^\prime(\xi,\xi^\prime)$. Due to Lemma~\ref{lem:indep}, we can do this by fixing $\xi^\prime=1$ (say) and taking the limit $\xi\to\infty$. To this end, we note that the asymptotic behaviour of the weight function in \eqref{wghtFunc} for $x\in \T_\xi^n$ and $y\in \T^m$ as $\xi\to\infty$ is given by
\begin{equation}
\label{DelnmAs}
\begin{split} 
\Delta_{n,m}(x,y;q,t) =  \frac{\Delta_n(x;q,t)\Delta_m(y;t,q)}{\prod_{i=1}^n\prod_{j=1}^m(-q^{-1/2}t^{1/2}x_i/y_j)}(1+O(1/\xi)) \\ 
= (-q^{-1/2}t^{1/2})^{-nm} e_n^m(x^{-1})e_m^n(y)\Delta_n(x;q,t)\Delta_m(y;t,q)(1+O(1/\xi))
\end{split} 
\end{equation}
where 
\begin{equation*} 
e^m_n(x^{-1}):= (x_1\cdots x_n)^{-m},\ \ \ e_m^n(y):=(y_1\cdots y_m)^{n}.
\end{equation*} 
Introducing the notation
\begin{equation} 
\label{muminmax} 
\mu_\text{min} := (\langle \lambda^\prime_1-n\rangle,\ldots,\langle \lambda^\prime_m-n\rangle), \ \ \ 
\mu_\text{max} := (\lambda^\prime_1,\ldots,\lambda^\prime_m),
\end{equation} 
where $\langle k\rangle:=\max(0,k)$, we use Lemma \ref{lem:SPExp} and \eqref{homP} 
 to deduce that, for $x\in \T_\xi^n$ and $y\in \T^m$, 
\begin{multline}
\label{SPprod}
SP_\lambda(x,y;q,t)SP_\lambda(x^{-1},y^{-1};q,t)\\
= (-q^{-1/2}t^{1/2})^{|\mu_\text{min}|+|\mu_\text{max}|}P_{\lambda/\mu_\text{min}^\prime}(x;q,t)P_{\lambda/\mu_\text{max}^\prime}(x^{-1};t,q)\\
\cdot Q_{\mu_\text{min}}(y;t,q)Q_{\mu_\text{max}}(y^{-1};t,q)+O(\xi^{|\mu_\text{max}|-|\mu_\text{min}|-1})
\end{multline}
as $\xi\to\infty$. From \eqref{DelnmAs}--\eqref{SPprod} we readily obtain, 
\begin{multline*}
\label{NnmExp}
N_{n,m}(\lambda;q,t) = (-q^{-1/2}t^{1/2})^{|\mu_\text{min}|+|\mu_\text{max}|-mn}\cdot \xi^{|\mu_\text{max}|-|\mu_\text{min}|-nm}\\
\cdot\langle P_{\lambda/\mu_\text{min}^\prime},e_n^mP_{\lambda/\mu_\text{max}^\prime}\rangle_{n;q,t}^\prime \langle e_m^nQ_{\mu_\text{min}},Q_{\mu_\text{max}}\rangle_{m;t,q}^\prime\\
+O(\xi^{|\mu_\text{max}|-|\mu_\text{min}|-nm-1})\ \ \ (\xi\to\infty)
\end{multline*}
using \eqref{homP}, with $\langle\cdot,\cdot\rangle_{n;q,t}$ in \eqref{prod0} 
(the factor $\xi^{|\mu_\text{max}|-|\mu_\text{min}|-nm}$ is due to the change of variables $x\to x/\xi$ transforming $\T_\xi^n\to \T_1^n=\T^n$). 
Since $|\mu_\text{max}|-|\mu_\text{min}|\leq nm$ with equality if and only if $(m^n)\subseteq\lambda$, the validity of \eqref{zNnm} immediately follows.

In the remaining cases $(m^n)\subseteq\lambda$,  
\begin{equation} 
\label{mumaxmin} 
\mu_\text{max}=\mu_\text{min}+(n^m),\ \ \ \mu_\text{min}=s(\lambda)
\end{equation} 
(cf.\ \eqref{es} and \eqref{muminmax}), and therefore 
\begin{equation*} 
(-q^{-1/2}t^{1/2})^{|\mu_\text{min}|+|\mu_\text{max}|-mn}\cdot \xi^{|\mu_\text{max}|-|\mu_\text{min}|-nm} = (t/q)^{|s(\lambda)|}
\end{equation*} 
independent of $\xi$. Thus, by taking the limit $\xi\to\infty$, we obtain 
\begin{equation} 
\label{Nnm0}
N_{n,m}(\lambda;q,t) =  (t/q)^{|s(\lambda)|} \langle P_{\lambda/\mu_\text{min}^\prime},e_n^mP_{\lambda/\mu_\text{max}^\prime}\rangle_{n;q,t}^\prime \langle e_m^nQ_{\mu_\text{min}},Q_{\mu_\text{max}}\rangle_{m;t,q}^\prime
\end{equation} 
with $\mu_\text{max}$ and $\mu_\text{min}$ in \eqref{mumaxmin}. 

We are left to compute the scalar products in \eqref{Nnm0}. We start with the second one: 
\begin{equation*} 
\begin{split} 
\langle e_m^nQ_{\mu_\text{min}},Q_{\mu_\text{max}}\rangle_{m;t,q}^\prime = & \langle e_m^nQ_{s(\lambda)},Q_{s(\lambda)+(n^m)}\rangle_{m;t,q}^\prime \\ 
= & b_{s(\lambda)}(t,q) b_{s(\lambda)+(n^m)}(t,q) \langle P_{s(\lambda)+(n^m)},P_{s(\lambda)+(n^m)}\rangle_{m;t,q}^\prime \\
= &  b_{s(\lambda)}(t,q) b_{s(\lambda)+(n^m)}(t,q)  N_m(s(\lambda);t,q)
\end{split} 
\end{equation*} 
by \eqref{Q}, \eqref{orth0} and \eqref{box}, using   
\begin{equation*} 
N_n(\lambda+(k^n);q,t)= N_n(\lambda;q,t)\quad (k\in\Z_{\geq 1}); 
\end{equation*} 
cf.~\eqref{Nn}. The first product can be computed in a similar manner, using \eqref{Pskew}, 
\begin{equation*} 
\label{prodP} 
\begin{split} 
\langle P_{\lambda/\mu_\text{min}^\prime},e_n^mP_{\lambda/\mu_\text{max}^\prime}\rangle_{n;q,t}^\prime = & \sum_{\nu_1,\nu_2} 
f^{\lambda^\prime}_{\mu_\text{min},\nu_1^\prime}(t,q) f^{\lambda^\prime}_{\mu_\text{max},\nu_2^\prime}(t,q)\langle P_{\nu_1},P_{\nu_2+(m^n)}\rangle^\prime_{n;q,t} \\
= & \sum_\nu f^{\lambda^\prime}_{\mu_\text{min},(\nu+(n^m))^\prime}(t,q) f^{\lambda^\prime}_{\mu_\text{max},\nu^\prime}(t,q) N_n(\nu+(m^n);q,t) \\
= &  \sum_\nu f^{\lambda^\prime}_{s(\lambda),(\nu+(m^n))^\prime}(t,q) f^{\lambda^\prime}_{s(\lambda)+(n^m),\nu^\prime}(t,q) N_n(\nu;q,t) . 
\end{split} 
\end{equation*} 
By Lemma~\ref{lem:f}, $f^{\lambda^\prime}_{s(\lambda),(\nu+(m^n)^\prime}$ and $f^{\lambda^\prime}_{s(\lambda)+(n^m),\nu^\prime}$ are non-zero only if 
\begin{equation} 
\label{cnu1} 
s(\lambda)^\prime\cup(\nu+(m^n))\leq \lambda 
\end{equation} 
and 
\begin{equation} 
\label{cnu2} 
\lambda\leq (s(\lambda)+(n^m))^\prime + \nu = (s(\lambda)^\prime\cup(m^n)) + \nu, 
\end{equation} 
respectively. The largest $\nu$ (in the sense of dominance ordering) satisfying \eqref{cnu1} is $\nu=e(\lambda)$ and, for this partition and only this,
\begin{equation*} 
s(\lambda)^\prime\cup(\nu+(m^n))=\lambda \ \ \text{ for } \nu=e(\lambda).  
\end{equation*} 
Similarly, the smallest $\nu$ satisfying \eqref{cnu2} is $\nu=e(\lambda)$ and, in this case and only then, 
\begin{equation*} 
(s(\lambda)^\prime\cup(m^n)) + \nu =\lambda  \ \ \text{ for } \nu=e(\lambda).  
\end{equation*} 
Thus, the $\nu$-sum in the last expression for $\langle P_{\lambda/\mu_\text{min}^\prime},e_n^mP_{\lambda/\mu_\text{max}^\prime}\rangle_{n;q,t}^\prime$ above has only a single non-zero term, namely $\nu=e(\lambda)$, and therefore 
\begin{equation*} 
\langle P_{\lambda/\mu_\text{min}^\prime},e_n^mP_{\lambda/\mu_\text{max}\prime}\rangle_{n;q,t}^\prime 
= f^{\lambda^\prime}_{s(\lambda),(e(\lambda)+(m^n))^\prime}(t,q) f^{\lambda^\prime}_{s(\lambda)+(n^m),e(\lambda)^\prime}(t,q) N_n(e(\lambda);q,t) .
\end{equation*} 
From Lemma~\ref{lem:f} we get 
\begin{equation*} 
f^{\lambda^\prime}_{s(\lambda),(e(\lambda)+(m^n))^\prime}(q,t)=1,\ \ \ 
f^{\lambda^\prime}_{s(\lambda)+(n^m),e(\lambda)^\prime}(t,q) =\frac{b_{s(\lambda)'\cup(m^n)}(q,t)b_{e(\lambda)}(q,t)}{b_\lambda(q,t)}. 
 \end{equation*} 
By inserting these results in \eqref{Nnm0} we obtain 
\begin{equation*} 
\begin{split} 
N_{n,m}(\lambda;q,t) =  (t/q)^{|s(\lambda)|} \frac{b_{s(\lambda)'\cup(m^n)}(q,t)b_{e(\lambda)}(q,t)}{b_\lambda(q,t)}N_n(e(\lambda);q,t) \\ \cdot 
b_{s(\lambda)}(t,q) b_{s(\lambda)+(n^m)}(t,q)  N_m(s(\lambda);t,q), 
\end{split} 
\end{equation*} 
and since 
\begin{equation*} 
b_{s(\lambda)'\cup(m^n)}(q,t)  = b_{(s(\lambda)+(m^n))^\prime}(q,t) = \frac1{b_{s(\lambda)+(m^n)}(t,q)}
\end{equation*} 
by  \eqref{blam}, we arrive at the result in \eqref{nzNnm}. \qed

\subsubsection{Part (a)}
Since $\Lambda_{n,m;q,t}$ is spanned by the super-Macdonald polynomials $SP_\lambda((x_1,\ldots,x_n),(y_1,\ldots,y_m);q,t)$, $\lambda\in H_{n,m}$, and \eqref{zNnm}--\eqref{nzNnm} clearly imply that the (quadratic) norms $N_{n,m}(\lambda;q,t)$ are real, it follows from Part (b) that our sesquilinear form $\langle\cdot,\cdot\rangle_{n,m;q,t}^\prime(\xi,\xi^\prime)$ \eqref{prod1} is Hermitian, i.e.
$$
\langle P,Q\rangle_{n,m;q,t}^\prime(\xi,\xi^\prime) = \overline{\langle Q,P\rangle_{n,m;q,t}^\prime(\xi,\xi^\prime)},\ \ \ (P,Q\in \Lambda_{n,m;q,t}).
$$

By invoking Lemma \ref{lem:inv}, we can thus infer
$$
\langle P,Q\rangle_{n,m;q,t}^\prime(\xi,\xi^\prime) = \overline{\langle Q,P\rangle_{n,m;q,t}^\prime(1/\xi,1/\xi^\prime)} = \langle P,Q\rangle_{n,m;q,t}^\prime(1/\xi,1/\xi^\prime).
$$
Due to the fact that $\xi/\xi^\prime<1/M(q,t)$ if and only if $(1/\xi)/(1/\xi^\prime)>M(q,t)$, Lemma \ref{lem:indep} implies that $\langle\cdot,\cdot\rangle_{n,m;q,t}^\prime(\xi,\xi^\prime)$ is independent of $\xi,\xi^\prime$ as long as \eqref{xixipAlt} (or equivalently \eqref{xixip}) is satisfied. This concludes the proof of Part (a) and hence the theorem.

\section{Conclusions and outlook}
\label{sec:final}
We introduced a Hermitian product $\langle\cdot,\cdot\rangle^\prime_{n,m;q,t}$, given by \eqref{prod0S}, on the algebra $\Lambda_{n,m;q,t}$, in which the super-Macdonald polynomials constitute an orthogonal basis (cf.~\cite[Theorem 5.6]{SV09a} and Theorem \ref{thm}), and we proved, in particular, that this product endows the factor space
$V_{n,m;q,t}=\Lambda_{n,m;q,t}/K_{n,m;q,t}$, where $K_{n,m;q,t}$ denotes the kernel of $\langle\cdot,\cdot\rangle^\prime_{n,m;q,t}$, with a 
Hilbert space structure.
Furthermore, we argued that these results provides the means for a quantum mechanical interpretation of the model defined by the deformed Macdonald operators $\cM_{n,m;q,t}$ and $\cM_{n,m;q^{-1},t^{-1}}$, cf.~\eqref{cMnm1}--\eqref{cMnm2}. This model describes two kinds of particles, and we proposed that they represent particles and anti-particles in an underlying relativistic quantum field theory, which is the same theory that inspired the Ruijsenaars models \cite{RS86,Rui01}. 

As mentioned in the introduction, from the quantum field theory point of view, it would be interesting to generalise our results to the elliptic case. 
The elliptic generalisation of the deformed Macdonald-Ruijsenaars operator is known from \cite{AHL14}. Specifically, rewriting the additive difference operator in Eq.~(70) in multiplicative form, we obtain
\begin{equation*}
\label{cMnm1e} 
M_{n,m;p,q,t} =  \frac{t^{1-n}q^m}{\theta(q;p)} \sum_{i=1}^n A_i(x,y;p,q,t)T_{q,x_i}+ \frac{t^{-n}q^{m-1}}{\theta(t^{-1};p)} \sum_{j=1}^m B_j(x,y;p,q,t)T_{t^{-1},y_j}
\end{equation*}
with the elliptic deformation parameter, $p$, in the range $0\leq p<1$ and coefficients
\begin{equation*} 
\label{cMnm2e}
\begin{split}
A_i(x,y;p,q,t) &= \prod_{i^\prime \neq i}^n \frac{\tet(tx_i/x_{i^\prime};p)}{\tet(x_i/x_{i^\prime};p)}\cdot \prod_{j=1}^m \frac{\tet(t^{1/2} x_i/q^{1/2} y_j;p)}{\tet(t^{1/2}q^{1/2} x_i/y_j;p)},\\
B_j(x,y;p,q,t) &= \prod_{j^\prime \neq j}^m \frac{\tet(q^{-1}y_j/y_{j^\prime};p)}{\tet(y_j/y_{j^\prime};p)}\cdot \prod_{i=1}^n \frac{\tet(q^{-1/2}t^{1/2} y_j/ x_i;p)}{\tet(q^{-1/2} y_j/t^{1/2} x_i;p)},
\end{split} 
\end{equation*} 
where $\tet(z;p) : = (z;p)_\infty(p/z;p)_\infty$. We note that $M_{n,m;p,q,t}$ reduces, up to an additive constant, to the deformed Macdonald--Ruijsenaars operator $M_{n,m;q,t}$ from \eqref{Mnm} in the trigonometric limit $p\to 0$. 
Moreover, the results in \cite{AHL14} suggest that a natural elliptic generalisation of our Hermitian product is as in \eqref{prod0S} but with the weight function
\begin{equation*}
\label{wghtFunce}
\Delta_{n,m}(x,y;p,q,t) = \frac{\Delta_n(x;p,q,t)\Delta_m(y;p,t,q)}{\prod_{i=1}^n\prod_{j=1}^m\tet(q^{-1/2}t^{1/2}x_i/y_j;p)\tet(q^{-1/2}t^{1/2}y_j/x_i;p)}, 
\end{equation*}
\begin{equation*}
\label{Delne}
\Delta_n(x;p,q,t) = \prod_{1\leq i\neq j\leq n}\frac{\Gamma(tx_i/x_j;p,q)}{\Gamma(x_i/x_j;p,q)},
\end{equation*}
where $\Gamma(z;p,q):= \prod_{k=0}^\infty (p^{k+1}q/z;q)_\infty/(p^k z;q)_\infty$ is the elliptic Gamma function; note that, in the limiting case $(n,m)=(n,0)$, this reduces to the operator and weight function of the elliptic Ruijsenaars model (see e.g.~\cite[Section~5]{Has97}).
However, at this point, very little is known about the eigenfunctions of the deformed elliptic Macdonald-Ruijsenaars operator $M_{n,m;p,q,t}$. 
In fact, even in the ordinary $m=0$ case, the understanding of these eigenfunctions is still not complete; see, however, \cite{Shi19,LNS20} for recent progress in this direction. 
Our results provide further motivation for any attempt at developing a theory of eigenfunctions at the deformed elliptic level, and generalising the results in \cite{Shi19,LNS20} to the deformed case could be an interesting starting point.

In the non-relativistic limit $q\to 1$, the trigonometric Ruijsenaars model reduces to the trigonometric Calogero-Sutherland model and, in this case, a quantum field theory formulation is known, which naturally includes the deformed models \cite{AL17}. Moreover, parts of this construction were extended recently to the elliptic case \cite[Section III.A]{BLL20}. Our results in this paper suggest that these quantum field theory results can be generalised to the Ruijsenaars case. A natural starting point would be a well-established quantum field theory description of the trigonometric Ruijsenaars model \cite{SKA92}, which allows for an elliptic generalisation \cite{FHHSY09}.

It is interesting to note that, while the deformed elliptic Calogero-Sutherland (eCS) system first appeared more than 10 years ago in a systematic search for kernel functions for eCS-type systems \cite{Lan10}, this very model recently appeared in the context of super-symmetric gauge theories \cite{Nek17,CKL20}. In our above-mentioned paper \cite{AHL14}, we obtained the deformed elliptic Ruijsenaars model by generalising the former results to the relativistic case; it would be interesting to also establish relativistic generalisations of the latter results.

Finally, we note that the result in Eqs.~\eqref{orth}--\eqref{nzNnm} remains true even for complex $q$ and $t$ such that $0<|q|<1$ and $0<|t|<1$, provided that the definition of $Q^*(x,y)$ in \eqref{prod} is changed to 
\begin{equation} 
Q^*(x,y) = Q(x^{-1},y^{-1})
\end{equation} 
(i.e.~no complex conjugation). 
However, then $\langle\cdot,\cdot\rangle^\prime_{n,m;q,t}$ is not sesquilinear and (in general) not positive (semi)definite, and thus no longer provides $V_{n,m;q,t}$ with a Hilbert space structure.

\appendix

\section{Conventions used by Sergeev and Veselov}
\label{app:SV} 
Here we explain the relation between the conventions for the super-Macdonald polynomials used in this paper, and the ones used by Sergeev and Veselov (SV) \cite{SV09a}. 
As will be made clear, it is easy to translate from one convention to the other.
Moreover, both conventions have their advantages and disadvantages.
More specificially, the advantages of our conventions are that the super-Macdonald polynomials are manifestly invariant under $(q,t)\to (q^{-1},t^{-1})$, and that Hilbert space adjungation agrees with what one would naively expect; cf.~\eqref{inv} vs.\ \eqref{Cinv} and \eqref{conj} vs.\ \eqref{Cconj}. The advantage of the SV-conventions is that factors $t^{\pm 1/2}$ and $q^{\pm 1/2}$ are avoided, and that some formulas look somewhat more symmetric; cf.~\eqref{cMnm2} vs.\ \eqref{CcMnm2}, \eqref{qinv} vs.\ \eqref{Cqinv}, and \eqref{SPreps} vs.\ \eqref{CSPreps}. 

The deformed power sums used in \cite{SV09a} are 
\begin{equation} 
\label{prX} 
p^{(\SV)}_{r}(x,y;q,t)= \sum_{i=1}^n x_i^r +\frac{1-q^r}{1-t^{-r}} \sum_{j=1}^m y_j^r \quad (r\in\Z_{\geq 1}); 
\end{equation} 
note that our $t^{-1}$ corresponds to $t$ in \cite{SV09a}.
The algebra endomorphism $\varphi^{(\SV)}_{n,m}$ defining the super-Macdonald polynomials is defined exactly as in \eqref{varphi}--\eqref{SPdef2} but with the deformed power sums in \eqref{prX}. 
Clearly, the deformed power sums in \eqref{prX} are obtained from ours in \eqref{pnmr} by the transformations $(x,y)\to (x,q^{1/2}t^{1/2}y)$, and this implies 
\begin{equation} 
\label{translate} 
SP^{(\SV)}_\lambda(x,y;q,t) = SP_\lambda(x;q^{1/2}t^{1/2}y;q,t). 
\end{equation} 

We recall that the coefficents $c_{\lambda\mu}(q,t)$ in \eqref{SPdef1} are invariant under the transformation $(q,t)\to (q^{-1},t^{-1})$; see \eqref{qinvtinv}. 
However, the arguments $(x,y)$ of $p^{(\SV)}_{r}(x,y;q,t)$ transform under this transformation to  $(x,q^{-1}t^{-1}y)$. 
Thus, 
\begin{equation} 
\label{Cinv}
SP^{(\SV)}_\lambda(x,y;q^{-1},t^{-1}) = SP^{(\SV)}_\lambda(x;q^{-1}t^{-1}y;q,t).
\end{equation} 

Our scalar product in the SV-conventions can be written as in \eqref{prod} but with a slightly altered weight function, another definition of conjugation, and different constraints on the radii $\xi,\xi^\prime>0$: 
\begin{equation} 
\Delta^{(\SV)}_{n,m}(x,y;q,t) =  \frac{\Delta_{n}(x;q,t) \Delta_{m}(y;t,q)}{\prod_{i=1}^n\prod_{j=1}^m (1-q^{-1}x_i/y_j)(1-ty_j/x_i)}
\end{equation} 
with $\Delta_{n}(x;q,t)$ in \eqref{Deln}, 
\begin{equation} 
\label{Cconj} 
Q^*(x,y) \equiv
\overline{Q(\bar x^{-1},q^{-1}t^{-1}\bar y^{-1})} , 
\end{equation} 
and 
\begin{equation} 
\xi/\xi^\prime < \min(q,t)\ \ \ \text{ or } \ \ \ \xi/\xi^\prime >\max(q,t) . 
\end{equation} 

For the convenience of the reader, we also give other important formulas in the SV-conventions. First, formulas for the coefficients in \eqref{cMnm2} defining the deformed Macdonald operators in \eqref{cMnm1}:
\begin{equation} 
\label{CcMnm2}
\begin{split} 
A^{(\SV)}_i(x,y;q,t) &= 
t^{n-1}\prod_{i^\prime \neq i}^n \frac{x_i - t^{-1}x_{i^\prime}}{x_i- x_{i^\prime}}\cdot \prod_{j=1}^m \frac{x_i - qy_j}{x_i - y_j}, 
\\
B^{(\SV)}_j(x,y;q,t) &= 
 q^{1-m}\prod_{j^\prime \neq j}^m \frac{y_j - qy_{j^\prime}}{y_j - y_{j^\prime}}\cdot \prod_{i=1}^n \frac{y_j - t^{-1} x_i}{y_j - x_i}.
\end{split} 
\end{equation}  
Second, the symmetry conditions in \eqref{qinv} that characterise the algebra $\Lambda^{(\SV)}_{n,m;q,t}$ spanned by the super-Macdonald polynomials: 
\begin{equation} 
\label{Cqinv} 
\left(T_{q,x_i}-T_{t^{-1},y_j} \right) P(x,y) =0\ \ \text{ at } \ \ x_i=y_j \ \ \ (\forall i,j).
\end{equation} 
Third, the representation of the super-Jack polynomials in \eqref{SPreps}: 
\begin{equation}
\label{CSPreps} 
SP^{(\SV)}_\lambda(x,y;q,t) = \sum_\mu (-t)^{|\mu|} P_{\lambda/\mu^\prime}(x;q,t)Q_\mu(y;t,q). 
\end{equation}

\section{Relativistic invariance of deformed Ruijsenaars model}
\label{app:relativistic} 
We show that the deformed Ruijsenaars model defined by the operators $\cM_{n,m;q,t}$ and $\cM_{n,m;q^{-1},t^{-1}}$ in \eqref{cMnm1}--\eqref{cMnm2} is relativistically invariant. 
We use the operator $M_{n,m;q,t}$ obtained from $\cM_{n,m;q,t}$ by dropping a constant term; see \eqref{Mnm} and \eqref{Id}. 

As explained by Ruijsenaars in his pioneering paper \cite[Eqs.\ (2.17)--(2.18)]{Rui87}, the Ruijsenaars models are relativistically invariant in the sense that certain operators $\hat S_1$ and $\hat S_{-1}$ provide a representation of the Lie algebra of the Poincar\'e group in 1+1 spacetime dimensions, i.e., they give operators $\hat H$ (Hamiltonian), $\hat P$ (momentum operator) and $\hat B$ (boost operator) satisfying the following commutation relations, 
\begin{equation} 
\label{HPB} 
[\hat H,\hat P]= 0,\quad [\hat H,\hat B]=\ii \hat P,\quad [\hat P,\hat B] = \ii \hat H,
\end{equation} 
(where, for simplicity, we use units such that $m=c=1$). 
This argument straightforwardly generalises to the deformed case:  {\em The operators 
\begin{equation*}
\begin{split} 
\hat H = & \Psi_0\frac12(M_{n,m;q,t} + M_{n,m;q^{-1},t^{-1}})\frac1{\Psi_0}, \\
\hat P = &  \Psi_0\frac12(M_{n,m;q,t} - M_{n,m;q^{-1},t^{-1}})\frac1{\Psi_0}, \\
\hat B =  & \Psi_0 B \frac1{\Psi_0}=B,\ \ \ B:=\ii\sum_{i=1}^n\frac{\log(x_i)}{\log(q)} - \ii \sum_{j=1}^m\frac{\log(y_j)}{\log(t)}, 
\end{split} 
\end{equation*}   
with $\Psi_0=\Psi_0(u,v;\beta,\gamma)$ in \eqref{Psi0}, satisfy the commutation relations in \eqref{HPB}.} 

To see this, we note that the operators $M_{n,m;q,t}$ and $M_{n,m;q^{-1},t^{-1}}$ commute on the space $\Lambda_{n,m;q,t}$ generated by the super-Macdonald polynomials.\footnote{This is implied by \eqref{inv} and the fact that the super-Macdonald polynomials are eigenfunctions of the operators $M_{n,m;q,t}$.} This is equivalent to the first relation in \eqref{HPB}. (From \cite[Lemma 3.1]{HLNR21} follows that commutativity as operators on $\Lambda_{n,m;q,t}$ implies commutativity as difference operators.)
 The second and the third relations in \eqref{HPB} are equivalent to  
\begin{equation*} 
[M_{n,m;q^{\pm 1},t^{\pm 1}},B] = \pm\ii M_{n,m;q^{\pm 1},t^{\pm 1}}, 
\end{equation*} 
which is easy to check using the following non-trivial commutation relations following from the definition of the shift operators, 
\begin{equation*} 
[T_{q^{\pm 1},x_i},\log(x_i)]= \pm \log(q)T_{q^{\pm},x_i}, \quad [T_{t^{\mp 1},y_j},\log(y_j)]= \mp \log(t)T_{t^{\mp 1},x_i}. 
\end{equation*}

\section{Proof details}
\label{app:proof} 
\subsection{Proof of Lemma~\ref{lem:f}}
\label{app:f} 
\begin{remark} 
We adapt a proof in the Jack polynomial case \cite[4.1 Proposition]{Sta89}.
\end{remark} 

The monomial functions in \eqref{mlam} satisfy  
\begin{equation*} 
m_\mu m_\nu = m_{\mu+\nu} + \text{lower order terms}
\end{equation*} 
for all partitions $\mu,\nu$, with ``lower order terms'' standing for a linear combination of $m_\lambda$ with $\lambda<\mu+\nu$. This and the definition of Macdonald functions $P_\lambda$ (triangular structure) imply 
\begin{equation*} 
P_\mu P_\nu = P_{\mu+\nu} + \text{lower order terms} .
\end{equation*} 
On the other hand, by definition \cite[Eq.\ (7.1')]{Mac95}, 
\begin{equation} 
\label{PPP}
P_\mu(x;q,t) P_\nu(x;q,t) = \sum_{\lambda} f^{\lambda}_{\mu\nu}(q,t) P_\lambda(x;q,t). 
\end{equation} 
Thus, by comparison, 
\begin{equation}
\label{f1}  
f^\lambda_{\mu\nu}(q,t)\equiv 0\ \ \text{ unless } \ \ \lambda\leq \mu+\nu, \ \ \ f^{\mu+\nu}_{\mu\nu}(q,t)=1. 
\end{equation} 
Substituting $(\lambda,\mu,\nu,q,t)\to(\lambda^\prime,\mu^\prime,\nu^\prime,t,q)$, and using that $f^\lambda_{\mu\nu}$ is non-zero only if $|\mu|+|\nu|=|\lambda|$, the latter is equivalent to 
\begin{equation*} 
f^{\lambda^\prime}_{\mu^\prime\nu^\prime}(t,q)\equiv 0\ \ \text{ unless } \ \ \mu\cup\nu\leq \lambda, \ \ \ f^{(\mu\cup\nu)^\prime}_{\mu^\prime\nu^\prime}(t,q)=1
\end{equation*} 
(since $\lambda^\prime\leq\mu^\prime$ is equivalent to $\mu\leq\lambda$ provided $|\mu|=|\lambda|$, and $\mu^\prime+\nu^\prime=(\mu\cup\nu)^\prime$ \cite{Mac95}). 
This proves the first half of the result. The second half is obtained from \eqref{f1} using the formula 
\begin{equation*} 
f^{\lambda^\prime}_{\mu^\prime\nu^\prime}(t,q)=\frac{b_\mu(q,t)b_\nu(q,t)}{b_{\lambda}(q,t)} f^\lambda_{\mu\nu}(q,t) 
\end{equation*} 
following from \eqref{PPP} by applying $\omega_{q,t}$ in \eqref{omqt}, renaming $(\lambda,\mu,\nu,q,t)\to(\lambda^\prime,\mu^\prime,\nu^\prime,t,q)$, and using \eqref{Q} and \eqref{duality}.
\qed

\subsection{Detail in the proof of Lemma~\ref{lem:SPExp}} 
\label{app:detail} 
For the convenience of the reader, we provide a self-contained proof of the fact that $P_{\lambda/\mu^\prime}((x_1,\ldots,x_n);q,t)\equiv 0$ if $\lambda^\prime_j-\mu_j>n$ for some $j\geq 1$; cf.~(7.15) in \cite[Section VI]{Mac95}.

By definition, 
\begin{equation*} 
P_{\lambda/\mu^\prime}((x_1,\ldots,x_n);q,t) = \sum_{\nu} f^{\lambda^\prime}_{\mu\nu^\prime}(t,q) P_{\nu}((x_1,\ldots,x_n);q,t) 
\end{equation*} 
where the sum on the right-hand side is only over partitions $\nu$ of length less or equal to $n$ (since $P_{\nu}((x_1,\ldots,x_n);q,t)\equiv 0$ otherwise), i.e., all partitions $\nu$ contributing to this sum satisfy 
\begin{equation} 
\nu_j^\prime\leq \nu_1^\prime=\ell(\nu)\leq n
\end{equation} 
for all $j=1,2,\ldots$. 

{By Lemma~\ref{lem:f}, the coefficients $ f^{\lambda^\prime}_{\mu\nu^\prime}(t,q)$ are non-zero only if $\mu^\prime\cup\nu\leq\lambda$, equivalent to $\lambda^\prime\leq\mu+\nu^\prime$, i.e., 
$$
\lambda^\prime_j-\mu_j\leq \nu^\prime_j
$$
for all $j=1,2,\ldots$. This implies the result. \qed

\subsection{Proof of \eqref{Id}}
\label{app:Id}
We consider the complex function 
$$
f(z) :=  \prod_{i=1}^n \frac{z - t^{-1/2} x_i}{z-t^{1/2} x_i} \cdot \prod_{j=1}^m  \frac{z - q^{1/2} y_j}{z-q^{-1/2} y_j}, 
$$
assuming fixed generic values for $x_i$ and $y_j$ (so that all poles of $f(z)$ are of order 1), and compute
$$
\lim_{\xi\to\infty}\oint_{|z| = \xi } \frac{dz}{2\pi\ii z}f(z)
$$
in two ways: first, using that $f(z)=1+O(1/z)$ as $|z|\to\infty$, which gives 1; second, invoking the residue theorem (for sufficiently large $\xi$), which gives the sum of all residues. This yields the identity 
\begin{multline*}
1 =  \sum_{i=1}^n \lim_{z\to t^{1/2}x_i} (z-t^{1/2}x_i) \frac{f(z)}{z} +   \sum_{j=1}^m \lim_{z\to q^{-1/2}y_j} (z-q^{-1/2}y_j)\frac{f(z)}{z} + \lim_{z\to 0}f(z)\\ 
=   \sum_{i=1}^n (1-t^{-1})t^{1-n} \prod_{i'\neq i}^n  \frac{tx_i - x_{i'}}{x_i-x_{i'}} \prod_{j=1}^m  \frac{t^{1/2} x_i - q^{1/2} y_j}{t^{1/2} x_i-q^{-1/2} y_j} 
\\ + \sum_{j=1}^m (1-q)q^{m-1}\prod_{i=1}^n \frac{q^{-1/2} y_j - t^{-1/2} x_i}{q^{-1/2} y_j-t^{1/2} x_i} \prod_{j'\neq j}^m  \frac{q^{-1} y_j - y_{j'}}{y_j -y_{j'}} 
+  t^{-n}q^m 
\end{multline*} 
which clearly is equivalent to the identity \eqref{Id}.

While we assumed generic $x_i$ and $y_j$ in this argument, it is clear by continuity that the result holds true for arbitrary complex  $x_i$ and $y_j$.  \qed

\subsection{Proof of \eqref{A1}}
\label{app:check} 
We give a detailed proof of the identity in \eqref{A1} assuming \eqref{xixip1} and \eqref{A2}. We recall that $i\in\{1,\ldots,n\}$ is fixed. 

As explained in the main text, we only need to show that the common integrand of the two integrals in \eqref{A1}  is an analytic function of the complex variable $x_i$ in the region $\xi\leq |x_i|\leq \xi/q$, provided the other variables are fixed as in \eqref{A2} and \eqref{xixip1} holds true. Since $P(x^{-1},y^{-1})  T_{q,x_i} \overline{Q}(x,y)$ is an analytic function of $(x,y)$ in $(\mathbb{C}^*)^n\times (\mathbb{C}^*)^m$, we only need to investigate the function $W_i^*(x,y) T_{q,x_i}W_i(x,y)$, which is equal to 
$$
  \prod_{i'\neq i}^n \frac{(x_{i'}/x_i;q)_\infty}{(tx_{i'}/x_i;q)_\infty}  \frac{(qx_i/x_{i'};q)_\infty}{(tqx_i/x_{i'};q)_\infty} \cdot \prod_{j=1}^m \frac{1}{(1-q^{-1/2}t^{1/2}y_j/x_i)}\frac{1}{(1-q^{1/2}t^{1/2}x_i/y_j)},
$$
cf.~\eqref{Wi}. We verify that no poles of the four types of factors in the latter expression are located in the pertinent $x_i$-region that we parametrise as follows: $|x_i|=\xi a/q$ with $q\leq a\leq 1$. 

The poles of the first type of factors are only encountered when $|q^k tx_{i'}/x_i| = q^k tq/a=1$ for $k\in\Z_{\geq 0}$, i.e., $a=q^{k+1}t<q$; there are no such poles for $q\leq a\leq 1$. The poles of the second type of factors are all located in the subsets $|q^k tqx_i/x_{i'}|=q^k t a=1$  for $k\in\Z_{\geq 0}$, i.e., $a=1/q^k t>1$; and again there are no such poles for $q\leq a\leq 1$. The poles of the third type of factors only occur for $|q^{-1/2}t^{1/2}y_j/x_i| = q^{1/2}t^{1/2}\xi^\prime/\xi a=1$, i.e., for $\xi/\xi^\prime=q^{1/2}t^{1/2}/a$; if $q\leq a\leq 1$, these poles occur for $q^{1/2}t^{1/2}\leq \xi/\xi^\prime\leq q^{-1/2}t^{1/2}$, and there are no such poles if \eqref{xixip1} holds true. Finally, the fourth type only have poles in the subsets $|q^{1/2}t^{1/2}x_i/y_j|=q^{-1/2}t^{1/2}a \xi/\xi^\prime=1$, i.e., for $\xi/\xi^\prime = q^{1/2}t^{-1/2}/a$; if $q\leq a\leq 1$, these poles occur for $q^{1/2}t^{-1/2}\leq \xi/\xi^\prime \leq q^{-1/2}t^{-1/2}$, and again there are no such poles if \eqref{xixip1} holds true. \qed

\section{The case $n=m=1$}
\label{app:nm1}
In this appendix, we consider the special case $n=m=1$, where we can verify by simple direct computations, that the sesquilinear form of Definition \ref{def:prod} is independent of the integration radii $\xi,\xi^\prime>0$ as long as $\xi/\xi^\prime>M(q,t):=\max_{\delta=\pm 1}\big(q^{\frac{\delta}{2}}t^{-\frac{\delta}{2}}\big)$ or 
 $\xi/\xi^\prime<\min_{\delta=\pm 1}\big(q^{\frac{\delta}{2}}t^{-\frac{\delta}{2}}\big)=1/M(q,t)$. 

In this case, the conditions \eqref{Lamnm}--\eqref{qinv} on elements in $\Lambda_{1,1;q,t}$, $P=P(x,y)$ with $(x,y)\in\C\times\C$, reduce to the symmetry condition \eqref{qinv} for $i=j=1$, i.e.,  
\begin{equation*} 
\big(T_{q,x}-T_{t^{-1},y}\big)P(x,y) = P(qx,y)-P(x,t^{-1}y)=0\  \text{ at } \  q^{1/2}x = t^{-1/2}y; 
\end{equation*} 
this can be written as 
\begin{equation} 
\label{qinv11} 
P(q^{1/2}x,t^{1/2}x)- P(q^{-1/2}x,t^{-1/2}x)= 0
\end{equation} 
(we inserted $y=q^{1/2}t^{1/2}x$ and renamed $q^{1/2}x\to x$). Moreover, the right-hand side of \eqref{prod} is of the form
\begin{equation*} 
I(\xi,\xi^\prime) := \oint_{|x|=\xi}\frac{\dd x}{2\pi \ii x}\oint_{|y|=\xi^\prime}\frac{\dd y}{2\pi \ii y}\frac{f(x,y)}{(1-q^{-1/2}t^{1/2}x/y)(1-q^{-1/2}t^{1/2}y/x)},
\end{equation*} 
where $\xi,\xi^\prime>0$ satisfy \eqref{xixip} and $f=PQ^*$ has the symmetry property \eqref{qinv11}. 
Clearly, by Cauchy's integral theorem, $I(\xi,\xi^\prime)$ is unchanged by continuous deformations of the integration contours as long as the singularities of the integrand are avoided; this is the case if \eqref{xixip} holds. 
Clearly, it suffices to show that $I(\xi,\xi^\prime)=I(\xi^\prime,\xi)$. 
For simplicity, we restrict attention to $t>q$, so that $M(q,t)=q^{-1/2}t^{1/2}$ and $1/M(q,t)=q^{1/2}t^{-1/2}$.  

Deforming the $y$-contour to $|y|=\xi$, we pick up a residue at the simple pole $y=q^{\frac{1}{2}}t^{-\frac{1}{2}}x$, and thus obtain 
\begin{equation}
\label{IxixipAlt} 
I(\xi,\xi^\prime) = I(\xi,\xi) + \frac{q}{q-t}\oint_{|x|=\xi}\frac{\dd x}{2\pi \ii x}f(x,q^{1/2}t^{-1/2}x).
\end{equation} 
Now deforming the $x$-contour in $I(\xi,\xi)$ to $|x|=\xi^\prime$, picking up a residue at $x=q^{\frac{1}{2}}t^{-\frac{1}{2}}y$, we find that
\begin{equation*} 
I(\xi,\xi^\prime) - I(\xi^\prime,\xi) = \frac{q}{q-t}\oint_{|x|=\xi}\frac{\dd x}{2\pi \ii x}\big(f(q^{-1/2}x,t^{-1/2}x) - f(q^{1/2}x,t^{1/2}x)\big)
\end{equation*} 
(where we have applied the scaling $x\to q^{-1/2}x$ to the former residue integral and $x\to t^{1/2}x$ to the latter, which, by Cauchy's theorem, does not alter their values). Since $f$ satisfies \eqref{qinv11}, the integral in the right-hand side is zero, which proves the claim. This highlights the importance of the symmetry condition \eqref{qinv}. 

It is interesting to note that if we were to deform the integration radii $\xi,\xi^\prime>0$ from the region $\xi/\xi^\prime>M(q,t)$ or $\xi/\xi^\prime<1/M(q,t)$ into the excluded region $1/M(q,t) < \xi/\xi^\prime<M(q,t)$, then our sesquilinear form $\langle P,Q\rangle_{1,1;q,t}^\prime$ would be changed by the addition of a residue term such as the integral in \eqref{IxixipAlt}. By considering specific examples, is readily seen that this spoils orthogonality of the super-Macdonald polynomials as well as non-negativity.

As we now sketch, the former conclusion can also be reached by demonstrating, through direct and straightforward computations, that the operator $\cM_{1,1;q,t}$, or equivalently,
$$
M_{1,1;q,t} = \frac{1}{1-q} A(x,y;q,t) T_{q,x} + \frac{1}{1-t^{-1}} B(x,y;q,t) T_{t^{-1},y},
$$
with coefficients
$$
A(x,y;q,t) = \frac{ t^{1/2} x - q^{1/2} y }{ t^{1/2} x - q^{-1/2} y } , \quad B(x,y;q,t) = \frac{q^{-1/2} y - t^{-1/2} x}{q^{-1/2} y - t^{1/2} x},
$$
is not self-adjoint when $\xi,\xi^\prime>0$ are chosen in the excluded region, cf.~Lemma \ref{lem:selfAdj}. 

For simplicity, we restrict attention to $\xi=\xi^\prime=1$ and $q<t$, but similar arguments apply for $q>t$ and other values of $\xi,\xi^\prime>0$ such that $1/M(q,t) < \xi/\xi^\prime<M(q,t)$. To start with, we consider
\begin{multline*}
\oint_{|x|=1}\frac{d x}{2\pi \ii x} \oint_{|y|=1}\frac{d y}{2\pi \ii y} \Delta_{1,1}(x,y;q,t) Q^{\ast}(x,y) A(x,y;q,t) T_{q,x}P(x,y) \\
 =\oint_{|x|=1}\frac{d x}{2\pi \ii x} \oint_{|y|=1}\frac{d y}{2\pi \ii y} \frac{Q^{\ast}(x,y) A(x,y;q,t) T_{q,x} P(x,y)}{(1-q^{-1/2} t^{1/2} x / y)(1- q^{-1/2} t^{1/2} y / x)}.
\end{multline*}
By shifting $x \to q^{-1} x$, the action of $T_{q,x}$ is transferred onto the polynomial $Q$ and we obtain
$$
\oint_{|x|=q}\frac{d x}{2\pi \ii x} \oint_{|y|= 1 }\frac{d y}{2\pi \ii y} \frac{  P(x,y) ( T_{q,x}Q)^\ast(x,y) A(q^{-1} x, y;q,t)}{(1-q^{-3/2} t^{1/2} x / y)(1- q^{1/2} t^{1/2} y / x)},
$$
where
$$
\frac{A(q^{-1} x, y;q,t)}{(1-q^{-3/2} t^{1/2} x / y)(1- q^{1/2} t^{1/2} y / x)} = \Delta_{1,1}(x,y;q,t)A^\ast(x,y;q,t).
$$
Deforming the $x$-contour back to $|x|=1$, we pick up residues at and only at the simple poles $x=q^{1/2}t^{-1/2}y$ and $x=q^{1/2}t^{1/2}y$, which yield the contribution
\begin{multline*}
\frac{q}{1-t}\oint_{|y|=1} \frac{dy}{2\pi \ii y}\Big(P(q^{1/2} t^{1/2} y , y) Q^{\ast}(q^{-1/2} t^{1/2} y , y)\\
 - P(q^{1/2} t^{-1/2} y, y) Q^\ast(q^{-1/2} t^{-1/2}y,y)\Big).
\end{multline*}

Rewriting the integral involving $B(x,y;q,t)T_{t^{-1},y}$ in a similar manner produces no further residue terms (under our parameter constraints). The upshot is that
\begin{multline*}
\oint_{|x|=1}\frac{d x}{2\pi \ii x} \oint_{|y|=1}\frac{d y}{2\pi \ii y} \Delta_{1,1}(x,y;q,t) Q^{\ast}(x,y) M_{1,1;q,t}P(x,y)\\ - \oint_{|x|=1}\frac{d x}{2\pi \ii x} \oint_{|y|=1}\frac{d y}{2\pi \ii y} \Delta_{1,1}(x,y;q,t)  P(x,y) (M_{1,1;q,t}Q)^{\ast}(x,y)\\
= \frac{q}{(1-q)(1-t)}\oint_{|y|=1} \frac{dy}{2\pi \ii y}\Big(P(q^{1/2} t^{1/2} y , y) Q^{\ast}(q^{-1/2} t^{1/2} y , y)\\
 - P(q^{1/2} t^{-1/2} y, y) Q^\ast(q^{-1/2} t^{-1/2}y,y)\Big),
\end{multline*}
where, in general, the right-hand side is non-zero.

\section*{Acknowledgments}
We would like to thank O.A.~Chalykh, M.~Noumi, A.N.~Sergeev, J.~Shiraishi, and A.P.~Veselov for helpful discussions. 
The work of F.A. was partially  carried out as a JSPS International Research Fellow and has been supported by the Japan Society for the Promotion of Science (Grant Nos.~P17768 and 17F17768). M.H.~acknowledges financial support from the Swedish Research Council (Reg.~nr.~2018-04291). E.L.~acknowledges support from the Swedish Research Council (Reg.~nr.~2016-05167), and by the Stiftelse Olle Engkvist Byggm\"astare (Contract 184-0573).

\bibliographystyle{amsalpha}

\end{document}